\documentclass[12pt]{amsart}
\usepackage{amssymb}
\usepackage{eucal}
\usepackage{amsmath}
\usepackage{amscd}
\usepackage[dvips]{color}
\usepackage{multicol}
\usepackage[all]{xy}           
\usepackage{graphicx}
\usepackage{color}
\usepackage{colordvi}
\usepackage{xspace}

\begin{document}  

\newcommand{\nc}{\newcommand}
\newcommand{\delete}[1]{}
\nc{\dfootnote}[1]{{}}          
\nc{\ffootnote}[1]{\dfootnote{#1}}
\nc{\mfootnote}[1]{\footnote{#1}} 
\nc{\todo}[1]{\tred{To do:} #1}

\nc{\mlabel}[1]{\label{#1}}  
\nc{\mcite}[1]{\cite{#1}}  
\nc{\mref}[1]{\ref{#1}}  
\nc{\mbibitem}[1]{\bibitem{#1}} 

\delete{
\nc{\mlabel}[1]{\label{#1}  
{\hfill \hspace{1cm}{\bf{{\ }\hfill(#1)}}}}
\nc{\mcite}[1]{\cite{#1}{{\bf{{\ }(#1)}}}}  
\nc{\mref}[1]{\ref{#1}{{\bf{{\ }(#1)}}}}  
\nc{\mbibitem}[1]{\bibitem[\bf #1]{#1}} 
}

\nc{\mkeep}[1]{\marginpar{{\bf #1}}} 

\newtheorem{theorem}{Theorem}[section]
\newtheorem{prop}[theorem]{Proposition}
\newtheorem{defn}[theorem]{Definition}
\newtheorem{lemma}[theorem]{Lemma}
\newtheorem{coro}[theorem]{Corollary}
\newtheorem{prop-def}[theorem]{Proposition-Definition}
\newtheorem{claim}{Claim}[section]
\newtheorem{remark}[theorem]{Remark}
\newtheorem{propprop}{Proposed Proposition}[section]
\newtheorem{conjecture}{Conjecture}
\newtheorem{exam}[theorem]{Example}
\newtheorem{assumption}{Assumption}
\newtheorem{condition}[theorem]{Assumption}

\renewcommand{\labelenumi}{{\rm(\roman{enumi})}}
\renewcommand{\theenumi}{\roman{enumi}}

\nc{\tred}[1]{\textcolor{red}{#1}}
\nc{\tblue}[1]{\textcolor{blue}{#1}}
\nc{\tgreen}[1]{\textcolor{green}{#1}}
\nc{\tpurple}[1]{\textcolor{purple}{#1}}
\nc{\btred}[1]{\textcolor{red}{\bf #1}}
\nc{\btblue}[1]{\textcolor{blue}{\bf #1}}
\nc{\btgreen}[1]{\textcolor{green}{\bf #1}}
\nc{\btpurple}[1]{\textcolor{purple}{\bf #1}}

\nc{\li}[1]{\textcolor{red}{Li:#1}}
\nc{\cm}[1]{\textcolor{blue}{Chengming: #1}}
\nc{\xiang}[1]{\textcolor{green}{Xiang: #1}}

\nc{\adec}{\check{;}} \nc{\aop}{\alpha}
\nc{\dftimes}{\widetilde{\otimes}} \nc{\dfl}{\succ}
\nc{\dfr}{\prec} \nc{\dfc}{\circ} \nc{\dfb}{\bullet}
\nc{\dft}{\star} \nc{\dfcf}{{\mathbf k}} \nc{\spr}{\cdot}
\nc{\twopr}{\circ} \nc{\tspr}{\star} \nc{\sempr}{\ast}
\nc{\disp}[1]{\displaystyle{#1}}
\nc{\bin}[2]{ (_{\stackrel{\scs{#1}}{\scs{#2}}})}  
\nc{\binc}[2]{ \left (\!\! \begin{array}{c} \scs{#1}\\
    \scs{#2} \end{array}\!\! \right )}  
\nc{\bincc}[2]{  \left ( {\scs{#1} \atop
    \vspace{-.5cm}\scs{#2}} \right )}  
\nc{\sarray}[2]{\begin{array}{c}#1 \vspace{.1cm}\\ \hline
    \vspace{-.35cm} \\ #2 \end{array}}
\nc{\bs}{\bar{S}} \nc{\dcup}{\stackrel{\bullet}{\cup}}
\nc{\dbigcup}{\stackrel{\bullet}{\bigcup}} \nc{\etree}{\big |}
\nc{\la}{\longrightarrow} \nc{\fe}{\'{e}} \nc{\rar}{\rightarrow}
\nc{\dar}{\downarrow} \nc{\dap}[1]{\downarrow
\rlap{$\scriptstyle{#1}$}} \nc{\uap}[1]{\uparrow
\rlap{$\scriptstyle{#1}$}} \nc{\defeq}{\stackrel{\rm def}{=}}
\nc{\dis}[1]{\displaystyle{#1}} \nc{\dotcup}{\,
\displaystyle{\bigcup^\bullet}\ } \nc{\sdotcup}{\tiny{
\displaystyle{\bigcup^\bullet}\ }} \nc{\hcm}{\ \hat{,}\ }
\nc{\hcirc}{\hat{\circ}} \nc{\hts}{\hat{\shpr}}
\nc{\lts}{\stackrel{\leftarrow}{\shpr}}
\nc{\rts}{\stackrel{\rightarrow}{\shpr}} \nc{\lleft}{[}
\nc{\lright}{]} \nc{\uni}[1]{\tilde{#1}} \nc{\wor}[1]{\check{#1}}
\nc{\free}[1]{\bar{#1}} \nc{\den}[1]{\check{#1}} \nc{\lrpa}{\wr}
\nc{\curlyl}{\left \{ \begin{array}{c} {} \\ {} \end{array}
    \right .  \!\!\!\!\!\!\!}
\nc{\curlyr}{ \!\!\!\!\!\!\!
    \left . \begin{array}{c} {} \\ {} \end{array}
    \right \} }
\nc{\leaf}{\ell}       
\nc{\longmid}{\left | \begin{array}{c} {} \\ {} \end{array}
    \right . \!\!\!\!\!\!\!}
\nc{\ot}{\otimes} \nc{\sot}{{\scriptstyle{\ot}}}
\nc{\otm}{\overline{\ot}} \nc{\ora}[1]{\stackrel{#1}{\rar}}
\nc{\ola}[1]{\stackrel{#1}{\la}}
\nc{\pltree}{\calt^\pl} \nc{\epltree}{\calt^{\pl,\NC}}
\nc{\rbpltree}{\calt^r} \nc{\scs}[1]{\scriptstyle{#1}}
\nc{\mrm}[1]{{\rm #1}}
\nc{\dirlim}{\displaystyle{\lim_{\longrightarrow}}\,}
\nc{\invlim}{\displaystyle{\lim_{\longleftarrow}}\,}
\nc{\mvp}{\vspace{0.5cm}} \nc{\svp}{\vspace{2cm}}
\nc{\vp}{\vspace{8cm}} \nc{\proofbegin}{\noindent{\bf Proof: }}
\nc{\proofend}{$\blacksquare$ \vspace{0.5cm}}
\nc{\freerbpl}{{F^{\mathrm RBPL}}}
\nc{\sha}{{\mbox{\cyr X}}}  
\nc{\ncsha}{{\mbox{\cyr X}^{\mathrm NC}}} \nc{\ncshao}{{\mbox{\cyr
X}^{\mathrm NC,\,0}}} \nc{\rpr}{\circ} \nc{\apr}{\cdot}
\nc{\dpr}{{\tiny\diamond}} \nc{\rprpm}{{\rpr}}
\nc{\shpr}{\diamond}    
\nc{\shprm}{\overline{\diamond}}    
\nc{\shpro}{\diamond^0}    
\nc{\shprr}{\diamond^r}     
\nc{\shpra}{\overline{\diamond}^r} \nc{\shpru}{\check{\diamond}}
\nc{\catpr}{\diamond_l} \nc{\rcatpr}{\diamond_r}
\nc{\lapr}{\diamond_a} \nc{\sqcupm}{\ot} \nc{\lepr}{\diamond_e}
\nc{\vep}{\varepsilon} \nc{\labs}{\mid\!} \nc{\rabs}{\!\mid}
\nc{\hsha}{\widehat{\sha}} \nc{\lsha}{\stackrel{\leftarrow}{\sha}}
\nc{\rsha}{\stackrel{\rightarrow}{\sha}} \nc{\lc}{\lfloor}
\nc{\rc}{\rfloor} \nc{\tpr}{\sqcup} \nc{\nctpr}{\vee}
\nc{\plpr}{\star} \nc{\rbplpr}{\bar{\plpr}} \nc{\sqmon}[1]{\langle
#1\rangle} \nc{\forest}{\calf} \nc{\ass}[1]{\alpha({#1})}
\nc{\altx}{\Lambda_X} \nc{\vecT}{\vec{T}} \nc{\onetree}{\bullet}
\nc{\Ao}{\check{A}} \nc{\seta}{\underline{\Ao}}
\nc{\deltaa}{\overline{\delta}} \nc{\trho}{\tilde{\rho}}

\nc{\mmbox}[1]{\mbox{\ #1\ }} \nc{\ann}{\mrm{ann}}
\nc{\Aut}{\mrm{Aut}} \nc{\can}{\mrm{can}} \nc{\twoalg}{{two-sided
algebra}\xspace} \nc{\bwt}{{mass}\xspace}
\nc{\bop}{{modification}\xspace} \nc{\ewt}{{mass}\xspace}
\nc{\ewts}{{masses}\xspace} \nc{\tto}{{extended}\xspace}
\nc{\Tto}{{Extended}\xspace} \nc{\tte}{{extended}\xspace}
\nc{\gyb}{{generalized}\xspace} \nc{\Gyb}{{Generalized}\xspace}
\nc{\MAYBE}{{EAYBE}\xspace} \nc{\GAYBE}{{GAYBE}\xspace}
\nc{\esym}{\vep} \nc{\colim}{\mrm{colim}} \nc{\Cont}{\mrm{Cont}}
\nc{\rchar}{\mrm{char}} \nc{\cok}{\mrm{coker}} \nc{\dtf}{{R-{\rm
tf}}} \nc{\dtor}{{R-{\rm tor}}}
\renewcommand{\det}{\mrm{det}}
\nc{\depth}{{\mrm d}} \nc{\Div}{{\mrm Div}} \nc{\End}{\mrm{End}}
\nc{\Ext}{\mrm{Ext}} \nc{\Fil}{\mrm{Fil}} \nc{\Frob}{\mrm{Frob}}
\nc{\Gal}{\mrm{Gal}} \nc{\GL}{\mrm{GL}} \nc{\Hom}{\mrm{Hom}}
\nc{\hsr}{\mrm{H}} \nc{\hpol}{\mrm{HP}} \nc{\id}{\mrm{id}}
\nc{\im}{\mrm{im}} \nc{\incl}{\mrm{incl}}
\nc{\length}{\mrm{length}} \nc{\LR}{\mrm{LR}} \nc{\mchar}{\rm
char} \nc{\NC}{\mrm{NC}} \nc{\mpart}{\mrm{part}}
\nc{\pl}{\mrm{PL}} \nc{\ql}{{\QQ_\ell}} \nc{\qp}{{\QQ_p}}
\nc{\rank}{\mrm{rank}} \nc{\rba}{\rm{RBA }} \nc{\rbas}{\rm{RBAs }}
\nc{\rbpl}{\mrm{RBPL}} \nc{\rbw}{\rm{RBW }} \nc{\rbws}{\rm{RBWs }}
\nc{\rcot}{\mrm{cot}} \nc{\rest}{\rm{controlled}\xspace}
\nc{\rdef}{\mrm{def}} \nc{\rdiv}{{\rm div}} \nc{\rtf}{{\rm tf}}
\nc{\rtor}{{\rm tor}} \nc{\res}{\mrm{res}} \nc{\SL}{\mrm{SL}}
\nc{\Spec}{\mrm{Spec}} \nc{\tor}{\mrm{tor}} \nc{\Tr}{\mrm{Tr}}
\nc{\mtr}{\mrm{sk}} \nc{\type}{{\bwt}\xspace}

\nc{\ab}{\mathbf{Ab}} \nc{\Alg}{\mathbf{Alg}}
\nc{\Algo}{\mathbf{Alg}^0} \nc{\Bax}{\mathbf{Bax}}
\nc{\Baxo}{\mathbf{Bax}^0} \nc{\RB}{\mathbf{RB}}
\nc{\RBo}{\mathbf{RB}^0} \nc{\BRB}{\mathbf{RB}}
\nc{\Dend}{\mathbf{DD}} \nc{\bfk}{{\bf k}} \nc{\bfone}{{\bf 1}}
\nc{\base}[1]{{a_{#1}}} \nc{\detail}{\marginpar{\bf More detail}
    \noindent{\bf Need more detail!}
    \svp}
\nc{\Diff}{\mathbf{Diff}} \nc{\gap}{\marginpar{\bf
Incomplete}\noindent{\bf Incomplete!!}
    \svp}
\nc{\FMod}{\mathbf{FMod}} \nc{\mset}{\mathbf{MSet}}
\nc{\rb}{\mathrm{RB}} \nc{\Int}{\mathbf{Int}}
\nc{\Mon}{\mathbf{Mon}}
\nc{\remarks}{\noindent{\bf Remarks: }}
\nc{\OS}{\mathbf{OS}} 
\nc{\Rep}{\mathbf{Rep}} \nc{\Rings}{\mathbf{Rings}}
\nc{\Sets}{\mathbf{Sets}} \nc{\DT}{\mathbf{DT}}

\nc{\BA}{{\mathbb A}} \nc{\CC}{{\mathbb C}} \nc{\DD}{{\mathbb D}}
\nc{\EE}{{\mathbb E}} \nc{\FF}{{\mathbb F}} \nc{\GG}{{\mathbb G}}
\nc{\HH}{{\mathbb H}} \nc{\LL}{{\mathbb L}} \nc{\NN}{{\mathbb N}}
\nc{\QQ}{{\mathbb Q}} \nc{\RR}{{\mathbb R}} \nc{\TT}{{\mathbb T}}
\nc{\VV}{{\mathbb V}} \nc{\ZZ}{{\mathbb Z}}


\nc{\calao}{{\mathcal A}} \nc{\cala}{{\mathcal A}}
\nc{\calc}{{\mathcal C}} \nc{\cald}{{\mathcal D}}
\nc{\cale}{{\mathcal E}} \nc{\calf}{{\mathcal F}}
\nc{\calfr}{{{\mathcal F}^{\,r}}} \nc{\calfo}{{\mathcal F}^0}
\nc{\calfro}{{\mathcal F}^{\,r,0}} \nc{\oF}{\overline{F}}
\nc{\calg}{{\mathcal G}} \nc{\calh}{{\mathcal H}}
\nc{\cali}{{\mathcal I}} \nc{\calj}{{\mathcal J}}
\nc{\call}{{\mathcal L}} \nc{\calm}{{\mathcal M}}
\nc{\caln}{{\mathcal N}} \nc{\calo}{{\mathcal O}}
\nc{\calp}{{\mathcal P}} \nc{\calr}{{\mathcal R}}
\nc{\calt}{{\mathcal T}} \nc{\caltr}{{\mathcal T}^{\,r}}
\nc{\calu}{{\mathcal U}} \nc{\calv}{{\mathcal V}}
\nc{\calw}{{\mathcal W}} \nc{\calx}{{\mathcal X}}
\nc{\CA}{\mathcal{A}}

\nc{\fraka}{{\mathfrak a}} \nc{\frakB}{{\mathfrak B}}
\nc{\frakb}{{\mathfrak b}} \nc{\frakd}{{\mathfrak d}}
\nc{\oD}{\overline{D}} \nc{\frakF}{{\mathfrak F}}
\nc{\frakg}{{\mathfrak g}} \nc{\frakk}{{\mathfrak k}}
\nc{\frakm}{{\mathfrak m}} \nc{\frakM}{{\mathfrak M}}
\nc{\frakMo}{{\mathfrak M}^0} \nc{\frakp}{{\mathfrak p}}
\nc{\frakS}{{\mathfrak S}} \nc{\frakSo}{{\mathfrak S}^0}
\nc{\fraks}{{\mathfrak s}} \nc{\os}{\overline{\fraks}}
\nc{\frakT}{{\mathfrak T}} \nc{\oT}{\overline{T}}
\nc{\frakX}{{\mathfrak X}} \nc{\frakXo}{{\mathfrak X}^0}
\nc{\frakx}{{\mathbf x}}
\nc{\frakTx}{\frakT}      
\nc{\frakTa}{\frakT^a}        
\nc{\frakTxo}{\frakTx^0}   
\nc{\caltao}{\calt^{a,0}}   
\nc{\ox}{\overline{\frakx}} \nc{\fraky}{{\mathfrak y}}
\nc{\frakz}{{\mathfrak z}} \nc{\oX}{\overline{X}}

\font\cyr=wncyr10

\nc{\redtext}[1]{\textcolor{red}{#1}}


\title[$\calo$-operators and Yang-Baxter equations]{$\calo$-operators on associative algebras and associative Yang-Baxter equations}

\author{Chengming Bai}
\address{Chern Institute of Mathematics\& LPMC, Nankai University, Tianjin 300071, P.R. China}
         \email{baicm@nankai.edu.cn}
\author{Li Guo}
\address{Department of Mathematics and Computer Science,
         Rutgers University,
         Newark, NJ 07102, U.S.A.}
\email{liguo@newark.rutgers.edu}
\author{Xiang Ni}
\address{Chern Institute of Mathematics \& LPMC, Nankai
University, Tianjin 300071, P.R.
China}\email{xiangn$_-$math@yahoo.cn}

\date{\today}


\begin{abstract}
We introduce the concept of an \tto $\calo$-operator that
generalizes the well-known concept of a Rota-Baxter operator. We
study the associative products coming from these operators and
establish the relationship between \tto $\calo$-operators and the
associative Yang-Baxter equation, \tto associative Yang-Baxter
equation and \gyb Yang-Baxter equation.
\end{abstract}


\maketitle

\tableofcontents

\setcounter{section}{0}
{\ }
\vspace{-1cm}

\section{Introduction}
This paper studies the connection between two concepts in quite different
contexts. One is that of a Rota-Baxter operator
originated from the probability study of Glenn Baxter~\mcite{Bax},
influenced by the combinatorial interests of Gian-Carlo Rota~\mcite{R1,R3} and applied broadly
in mathematics and physics in recent years~\mcite{CK2,EG1,EGK,EGM,GK1,GZ1}. The
other concept is that of a solution of the associative Yang-Baxter
equations which is an analogue of the classical Yang-Baxter
equation in mathematical physics, named after the well-known
physicists~\mcite{Ba,Ya}.  A connection between these two objects
were first established by Aguiar~\mcite{Ag1,Ag2} who showed that a
solution of the associative Yang-Baxter equation gives rise to a
Rota-Baxter operator of weight zero. Such connection has been
pursued further in several subsequent
papers~\mcite{Ag3,Bai2,E1}.

We revisit this connection with an alternative approach in order to gain better understanding of the relationship between these two concepts. On one hand we generalize the concept of Rota-Baxter operators to that of $\calo$-operators and further to \tto $\calo$-operators. On the other hand we investigate the operator properties of the associative Yang-Baxter equation motivated by the study in the Lie algebra case~\mcite{Bai1}. The $\calo$-operator is a relative version of the Rota-Baxter operator and, in the context of Lie algebras, was defined by Kupershmidt~\mcite{Ku} in the study of Yang-Baxter equations and can be traced back to Bordemann~\mcite{Bo} in integrable systems. Through this approach, we show that the operator property of solutions of the associative Yang-Baxter equation is to a large extent characterized by $\calo$-operators.

{\bf Notations: } In this paper, $\bfk$ denotes a field
and is often taken to have characteristic not equal
to 2.
 By an algebra we mean an associative (not
necessarily unitary) $\bfk$-algebra, unless otherwise stated.

\subsection{Rota-Baxter algebras and Yang-Baxter equations}
\mlabel{ss:ayb}


We recall concepts and relations that motivated our study.
\begin{defn}{\rm Let $R$ be a $\bfk$-algebra and let $\lambda\in \bfk$ be given. If a $\bfk$-linear map $P:R\rightarrow R$ satisfies the {\bf Rota-Baxter relation}:
\begin{equation}
P(x)P(y)=P(P(x)y)+P(xP(y))+ \lambda P(xy),
\quad \forall x,y \in R,
\mlabel{eq:rbo}
\end{equation}
then $P$ is called a {\bf Rota-Baxter
operator of weight $\lambda$} and $(R,P)$ is called a {\bf Rota-Baxter
algebra of weight $\lambda$.}
}
\mlabel{de:rb}
\end{defn}

For simplicity, we will only discuss the case of Rota-Baxter operators of weight zero in the introduction.

Note that the relation~(\mref{eq:rbo}) still makes sense when $R$ is
replaced by a $\bfk$-module with a binary operation. When the binary
operation is the Lie bracket and if in addition, the Lie algebra is
equipped with a nondegenerate symmetric invariant bilinear form,
then a skew-symmetric solution of the {\bf classical
Yang-Baxter equation}
\begin{equation}
[r_{12},r_{13}]+[r_{12},r_{23}]+[r_{13},r_{23}]=0.
\mlabel{eq:cyb}
\end{equation}
is just a Rota-Baxter operator of weight zero.
We refer the reader to~\mcite{Bai1,E2,Se} for further details.

\smallskip

We consider the following associative analogue of the classical Yang-Baxter equation~(\mref{eq:cyb}).
\begin{defn}
{\rm Let $A$ be a $\bfk$-algebra.
An element $r\in A\otimes A$ is called a {\bf solution of the associative Yang-Baxter equation in $A$} if it satisfies the relation
\begin{equation}
r_{12} r_{13}+r_{13} r_{23}-r_{23} r_{12}=0,
\mlabel{eq:aybe}
\end{equation}
called the {\bf associative Yang-Baxter equation (AYBE)}.
Here, for $r=\sum\limits_i
a_i\otimes b_i\in A\ot A$, denote
\begin{equation}
r_{12}=\sum\limits_i a_i\otimes b_i\otimes 1,\;\;
r_{13}=\sum\limits_i a_i\otimes1\otimes b_i
,\;\;r_{23}=\sum\limits_i 1\otimes a_i\otimes b_i. \mlabel{eq:r12}
\end{equation}
}
\mlabel{de:ayb}
\end{defn}
Both Eq.~(\mref{eq:aybe}) and another associative analogue of the classical Yang-Baxter equation~(\mref{eq:cyb})
\begin{equation}
r_{13}r_{12}-r_{12}r_{23}+r_{23}r_{13}=0.
\mlabel{eq:aayb}
\end{equation}
were introduced by Aguiar~\mcite{Ag1,Ag2,Ag3}. In fact,
Eq.~(\mref{eq:aybe}) is just Eq.~(\mref{eq:aayb}) in the opposite
algebra~\mcite{Ag3}. When $r$ is skew-symmetric it is easy to
see that Eq.~(\mref{eq:aybe}) comes from Eq.~(\mref{eq:aayb}) under
the operation $\sigma_{13} (x\otimes y\otimes z)=z\otimes y\otimes
x$. While Eq.~(\mref{eq:aayb}) was emphasized
in~\mcite{Ag1,Ag2,Ag3}, we will work with Eq.~(\mref{eq:aybe})
for notational convenience and to be consistent with some of the earlier works on connections with antisymmetric infinitesimal
bialgebras~\mcite{Bai2} and associative D-bialgebras \mcite{Z}.

\begin{theorem} {\bf (Aguiar~\mcite{Ag2})} Let $A$ be a $\bfk$-algebra. For a solution
$r=\sum\limits_i a_i\otimes b_i\in A\ot A$ of
Eq.~(\mref{eq:aayb}) in $A$, the map
\begin{equation}
P: A\to A, \;\; P(x)=\sum_i a_ixb_i,\;\;\forall x\in A,
\notag 
\end{equation}
defines a Rota-Baxter operator of weight zero on $A$.
\mlabel{thm:ag}
\end{theorem}
The theorem is obtained by replacing the tensor symbols in
 \begin{equation}
r_{13}r_{12}-r_{12}r_{23}+r_{23}r_{13} =\sum_{i,j}a_ia_j\otimes
b_j\otimes b_i-\sum_{i,j}a_i\otimes b_ia_j\otimes
b_j+\sum_{i,j}a_j\otimes a_i\otimes b_ib_j=0, \notag
\end{equation}
by $x$ and $y$ in $A$.

\subsection{$\calo$-operators}
In this paper, we introduce the concept of an \tto
$\calo$-operator as a generalization of the concept of a
Rota-Baxter operator and the associative analogue of an
$\calo$-operator on a Lie algebra. We then extend the connections
of Rota-Baxter algebras with associative Yang-Baxter equations to
those of $\calo$-operators. This study is motivated by the
relationship between $\calo$-operator and the classical Yang-Baxter
equation in Lie algebras~\mcite{Bai1,Bo,Ku}

Let $(A,\spr)$ be a $\bfk$-algebra. Let $(V,\ell,r)$ be an
$A$-bimodule, consisting of a compatible pair of a left $A$-module
$(V,\ell)$ given by $\ell: A\to \End_\bfk(V)$ and a right $A$-module
$(V,r)$ given by $r:A\to \End_\bfk(V)$ (see Section~\mref{ss:preliminaries} for
the precise definition). Fix a $\kappa\in \bfk$. A pair
$(\alpha,\beta)$ of linear maps $\alpha, \beta:V\to A$ is called
an {\bf \tto $\calo$-operator with \bop $\beta$ of \bwt $\kappa$}  if
\begin{equation}
\begin{array}{l}
\kappa \ell(\beta(u))v = \kappa u r(\beta(v)),\\
\alpha(u)\apr\alpha(v)-\alpha(l(\alpha(u))v+ur(\alpha(v)))
=\kappa \beta(u)\apr\beta(v),\;\;\forall u,v\in V.
\end{array}
\notag 
\end{equation}
When $\beta=0$ or $\kappa=0$,  we obtain the concept of an {\bf $\calo$-operator
$\aop$} satisfying
\begin{equation}
\aop(u)\apr
\aop(v)=\aop(\ell(\aop(u))v)+\aop(ur(\aop(v))),\;\;\forall u,v\in V.
\mlabel{eq:aop0}
\end{equation}
When $V$ is taken to be the $A$-bimodule
$(A,L,R)$ where $L,R: A\to \End_\bfk(A)$ are given by the left and right multiplications, an $\calo$-operator $\aop: V\to A$ of weight zero is just a Rota-Baxter
operator of weight zero.
To illustrate the close relationship between $
\calo$-operators and solutions of the AYBE~(\mref{eq:aybe}), we give the following reformulation of a part
of Corollary~\mref{co:aybea}. See Section~\mref{sec:aybe} for general
cases.

Let $\bfk$ be a field whose characteristic is not 2. Let $A$ be a $\bfk$-algebra that we for now assume to have finite dimension over $\bfk$. Let $\sigma: A\ot A \to A\ot A, a\ot b \mapsto b\ot a,$ be the switch
operator and let $t:\Hom_\bfk(A^*,A)\to \Hom_\bfk(A^*,A)$ be the transpose
operator. Then the natural bijection
$$ \phi: A\ot A \to \Hom_\bfk(A^*,\bfk)\ot A\to \Hom_\bfk(A^*,A)$$
is compatible with the operators $\sigma$ and $t$. Let
$\text{Sym}^2(A\ot A)$ and $\text{Alt}^2(A\ot A)$ (resp.
$\Hom_\bfk(A^*,A)_+$ and $\Hom_\bfk(A^*,A)_-$) be the eigenspaces for the
eigenvalues $1$ and $-1$ of $\sigma$ on $A\ot A$ (resp. of $t$ on
$\Hom_\bfk(A^*,A)$). Thus we have the commutative diagram of bijective
linear maps:
\begin{equation}
\xymatrix{ A\ot A\  \ar^{\phi}@{>->>}[r] \ar@{>->>}[d] & \Hom_\bfk(A^*,A) \ar@{>->>}[d]  \\
\text{Alt}^2 (A\ot A) \oplus \text{Sym}^2(A\ot A)\ \ar^{\phi}@{>->>}[r] & \Hom_\bfk(A^*,A)_- \oplus \Hom_\bfk(A^*,A)_+
}
\mlabel{eq:corr}
\end{equation}
preserving the factorizations. Let $\Hom_{bim}(A^*,A)_+$ be the
subset of $\Hom_\bfk(A^*,A)_+$ consisting of $A$-bimodule homomorphisms
from $A^*$ to $A$ both of which are equipped with the natural
$A$-bimodule structures. Denote $\text{Sym}^2_{bim}(A\ot
A):=\phi^{-1}(\Hom_{bim}(A^*,A)_+)\subseteq \text{Sym}^2(A\ot A)$.
Then we have (Corollary~\mref{co:aybea})

\begin{theorem}
An element $r=(r_-,r_+)\in \text{Alt}^2 (A\ot A) \oplus
\text{Sym}^2_{bim}(A\ot A)$ is a solution of the AYBE~(\mref{eq:aybe}) if and only
if the pair $\phi(r)=(\phi(r)_-, \phi(r)_+)=(\phi(r_-),\phi(r_+))$
is an \tto $\calo$-operator with \bop $\phi(r_+)$ of \bwt
$\kappa=-1$. In particular, when $r_+$ is zero, an element
$r=(r_-,0)=r_-\in \text{Alt}^2 (A\ot A) $ is a solution of the
AYBE if and only if the pair $\phi(r)=(\phi(r)_-, 0)=\phi(r_-)$ is
an $\calo$-operator of weight zero given by Eq.~$($\mref{eq:aop0}$)$
when $(V,\ell,r)$ is the dual bimodule $(A^*,R^*,L^*)$ of $(A,L,R)$. \mlabel{thm:corr}
\end{theorem}

Let $\calm\calo(A^*,A)$ denote the set of \tto $\calo$-operators $(\alpha,\beta)$ from $A^*$ to $A$ of \bwt $\kappa=-1$.
Let $\calo(A^*,A)$ denote the set of $\calo$-operators $\alpha:A^*\to A$ of weight $0$.
Let $\text{AYB}(A)$ denote the set of solutions of the AYBE~(\mref{eq:aybe}) in $A$.
Let $\text{SAYB}(A)$ denote the set of skew-symmetric solutions of the AYBE~(\mref{eq:aybe}) in $A$.
Then Theorem~\mref{thm:corr} means that the bijection in Eq.~(\mref{eq:corr}) restricts to bijections in the following commutative diagram.

{\small
\begin{equation}
\xymatrix{ \text{Alt}^2 (A\ot A) \oplus \text{Sym}^2_{bim}(A\ot A) \ar^{\phi}@{>->>}[r] & \Hom_\bfk(A^*,A)_- \oplus \Hom_{bim}(A^*,A)_+  \\
\text{AYB}(A) \bigcap\! \Big(\text{Alt}^2 (A\ot\!\! A) \oplus
\text{Sym}^2_{bim}(A\ot\! A)\Big)
\ar^{\phi}@{>->>}[r] \ar@{^{(}->}[u] & \calm\calo(A^*\!,A) \bigcap\! \Big(\Hom_\bfk(A^*\!,A)_- \oplus \Hom_{bim}(A^*\!,A)_+\Big) \ar@{^{(}->}[u] \\
\text{SAYB}(A) \ar@{^{(}->}[u] \ar^{\phi}@{>->>}[r] &
\calo(A^*,A)\bigcap \Hom_\bfk(A^*,A)_-  \ar@{^{(}->}[u] }
\notag
\end{equation}
}

\subsection{Layout of the paper}

In Section~\mref{sec:moop}, the
concept of an \tto $\calo$-operator is introduced and its connection
with the associativity of certain products is studied.
Section~\mref{sec:aybe} establishes the relationship of \tto
$\calo$-operators with associative and \tte associative Yang-Baxter equations.
Section~\mref{sec:maybe} introduces the concept of the \gyb associative Yang-Baxter equation
(\GAYBE) and considers its relationship with \tto $\calo$-operators.

\medskip

\noindent {\bf Acknowledgements: }
The first author would like to thank the support by the National Natural Science
Foundation of China (10621101), NKBRPC (2006CB805905) and SRFDP
(200800550015). The second author thanks the NSF grant DMS 0505445
for support and thanks the Chern Institute of Mathematics at Nankai
University for hospitality.

\section{$\calo$-operators and \tto $\calo$-operators}
\mlabel{sec:moop} We give background notations in Section~\mref{ss:preliminaries} before introducing the concept of an \tto $\calo$-operator in Section~\mref{ss:moop}. We then show in Section~\mref{ss:oass} and \mref{ss:rb} that \tto $\calo$-operators can be characterized by the associativity of a multiplication derived from this operator.

\subsection{Bimodules, $A$-bimodule $\bfk$-algebras and matched pairs of algebras}
\mlabel{ss:preliminaries}
We first recall the concept of a bimodule.
\begin{defn}
Let $(A,\apr)$ be a $\bfk$-algebra.
\begin{enumerate}
\item
An {\bf $A$-bimodule} is a $\bfk$-module $V$, together with linear
maps $\ell, r:A\rightarrow \End_\bfk(V)$, such that $(V,\ell)$
defines a left $A$-module, $(V,r)$ defines a right $A$-module and
the two module structures on $V$ are compatible in the sense that
\begin{equation}
(\ell(x)v)r(y)=\ell(x)(vr(y)),\; \forall\; x,y\in A, v\in V.
\notag 
\end{equation}
If we want to be more precise, we also denote an $A$-bimodule $V$ by
the triple $(V,\ell,r)$.
\item
A homomorphism between two $A$-bimodules $(V_1,\ell_1,r_1)$ and $(V_2,\ell_2,r_2)$ is a $\bfk$-linear map $g: V_1\to V_2$ such that
\begin{equation}
g(\ell_1(x)v)=\ell_2(x) g(v),\quad g(vr_1(x))=g(v)r_2(x), \quad \forall\, x\in A, v\in V_1.
\notag 
\end{equation}
\end{enumerate}
\mlabel{de:bim}
\end{defn}

For a $k$-algebra $A$ and $x\in A$, define the left and right actions
$$L(x): A\to A,\ L(x)y=xy\,; \quad R(x): A\to A,\  yR(x)=yx, \quad y\in A.$$
Further define
$$L=L_A:A\rightarrow \End_\bfk(A),\ x\mapsto L(x); \quad
R=R_A:A\rightarrow \End_\bfk(A),\ x\mapsto R(x), \quad x\in A.$$
Obviously, $(A,L,R)$ is an $A$-bimodule.

For a $\bfk$-module $V$, let $V^*:=\Hom_\bfk(V,\bfk)$ denote the
dual $\bfk$-module. Denote the usual pairing between $V^*$ and $V$
by
\begin{equation}
\langle\ ,\ \rangle: V^* \times V \to \bfk, \langle u^*, v \rangle
= u^*(v), \; \forall u^*\in V^*, v\in V.
\notag 
\end{equation}

\begin{prop} $($\mcite{Bai2}$)$ Let $A$ be a $\bfk$-algebra
and let $(V,\ell,r)$ be an $A$-bimodule. Define the linear maps
$\ell^*,r^*:A\rightarrow \End_\bfk(V^*)$ by
\begin{equation}
\langle u^*\ell^*(x),v\rangle=\langle u^*,\ell(x)v\rangle,\; \langle
r^*(x)u^*,v\rangle=\langle u^*,vr(x)\rangle,\;\forall x\in A, u^*\in
V^*, v\in V, \mlabel{eq:dual}
\end{equation}
respectively. Then $(V^*,r^*,\ell^*)$ is an $A$-bimodule, called the dual bimodule of $(V,\ell,r)$.
\mlabel{pp:dual}
\end{prop}

Let $(A^*,R^*,L^*)$ denote the dual $A$-bimodule of the $A$-bimodule
$(A,L,R)$.

\medskip

We next extend the concept of a bimodule to that of an
$A$-bimodule algebra by replacing the $\bfk$-module $V$ by a
$\bfk$-algebra $R$.

\begin{defn}{\rm
Let $(A,\apr)$ be a $\bfk$-algebra with multiplication $\apr$ and
let $(R,\rpr)$ be a $\bfk$-algebra with multiplication $\rpr$. Let
$\ell, r:A\rightarrow \End_\bfk(R)$ be two linear maps. We call $R$
(or the triple $(R,\ell,r)$ or the quadruple $(R,\rpr,\ell,r)$) an
{\bf $A$-bimodule $\bfk$-algebra} if $(R,\ell,r)$ is an $A$-bimodule
that is compatible with the multiplication $\rpr$ on $R$. More
precisely, we have
\begin{eqnarray}
\ell(x\apr y)v&=&\ell(x)(\ell(y)v),\;\ell(x)(v\rpr w)=(\ell(x)v)\rpr
w,
\mlabel{eq:twoalg1}\\
vr(x\apr y)&=&(vr(x))r(y),\; (v\rpr w)r(x)=v\rpr (wr(x)),\;
\mlabel{eq:twoalg2}\\
(\ell(x)v)r(y)&=&\ell(x)(vr(y)),\; (vr(x))\rpr w=v\rpr
(\ell(x)w),\; \forall\; x,y\in A, v,w\in R. \mlabel{eq:twoalg3}
\end{eqnarray}
}
\mlabel{de:bimal}
\end{defn}

Obviously, for any $\bfk$-algebra $(A,\apr)$, $(A,\apr, L,R)$ is
an $A$-bimodule $\bfk$-algebra.
Note that an $A$-bimodule $\bfk$-algebra $R$ need not be a left or right
$A$-algebra since we do not assume that $A\cdot 1$ is in the center of $R$.
For example, the $A$-bimodule $\bfk$-algebra $(A,L,R)$ is an $A$-algebra if and only if $A$ is a commutative ring.

The concept of an $A$-bimodule $\bfk$-algebra can be further
generalized to that of a matched pair introduced in~\mcite{Bai2}.
\begin{defn}
A {\bf matched pair} is a pair of associative algebras $(A,\apr)$
and $(B,\rpr)$ together with linear maps $\ell_A,r_A:A\rightarrow
\End_\bfk(B)$ and $\ell_B,r_B:B\rightarrow \End_\bfk(A)$ such that
$(B,\ell_A,r_A)$ is an $A$-bimodule and $(A,\ell_B,r_B)$ is a
$B$-bimodule, satisfying the following conditions:
{\allowdisplaybreaks
\begin{eqnarray}
\ell_A(x)(a\rpr b)&=& \ell_A(x\,r_B(a))b+(\ell_A(x)a)\rpr b;\mlabel{eq:2.4}\\
(a\rpr b)\,r_A(x)&=& a\,r_A(\ell_B(b)x)+a\rpr (b\,r_A(x));\mlabel{eq:2.5}\\
\ell_B(a)(x\apr y)&=&\ell_B(a\,r_A(x))y+(\ell_B(a)x)\apr y;\mlabel{eq:2.6}\\
(x\apr y)\,r_B(a)&=&x\,r_B(\ell_A(y)a)+x\apr (y\,r_B(a));\mlabel{eq:2.7}\\
\ell_A(\ell_B(a)x)b+(a\,r_A(x))\rpr b&=&a\,r_A(x\,r_B(b))+a\rpr
(\ell_A(x)b)=0;\mlabel{eq:2.8}\\
\ell_B(\ell_A(x)a)y+(x\,r_B(a))\apr y&=&x\,r_B(a\,r_A(y))+x\apr
(\ell_B(a)y)=0,\mlabel{eq:2.9}
\end{eqnarray}
}
for any $x,y\in A,a,b\in B$. \mlabel{de:mp}
\end{defn}
Matching pairs are naturally related to algebraic structures on a
direct sum of algebras.
\begin{theorem}$($\mcite{Bai2}$)$  Let $(A,B,\ell_A,r_A,\ell_B,r_B)$ be a matched pair.  Then there is
an algebra structure on the vector space $A\oplus B$ given by
\begin{equation}
(x+a)*(y+b)=(x\,\apr\, y+\ell_B(a)y+x\,r_B(b))+(a\rpr
b+\ell_A(x)b+a\,r_A(y)),\;\;\forall x,y\in A,a,b\in B. \mlabel{eq:2.10}
\end{equation}
We denote this associative algebra by
$A\bowtie^{\ell_A,r_A}_{\ell_B,r_B}B$ or simply $A\bowtie B$. On the other
hand, every associative algebra which is the direct sum (as vector
spaces) of two subalgebras can be obtained in this way.
\mlabel{thm:mp}
\end{theorem}
\delete{
\begin{proof}
For any $x,y,z\in A$ and $a,b,c\in B$, we have
\begin{eqnarray*}
&&[(x+a)*(y+b)]*(z+c)-(x+a)*[(y+b)*(z+c)]\\
&=&[(x\apr y)\apr z-x\apr (y\apr z)]+ [(x*b)*c-x*(b\rpr
c)]+[(a*y)*c-a*(y*c)]\\
&&+[(x\apr y)*c-x*(y*c)]+[(x*b)*z-x*(b*z)]+[(a\rpr b)\rpr
c-a\rpr
(b\rpr c)]\\
&&+[(a*y)*z-a*(y\apr z)]+[(a\rpr b)*z-a*(b*z)]
\end{eqnarray*}
Therefore, Eq.~(\mref{eq:2.10}) defines a $\bfk$-algebra
structure on $A\oplus B$ if and only if the following equations are
satisfied:
\begin{eqnarray*}
(x\apr y)\apr z=x\apr (y\apr z)\;&\Longleftrightarrow&\;\;A\;{\rm is}\;{\rm an}\;{\rm associative}\;{\rm algebra};\\
(x\apr y)*c=x*(y*c)\;&\Longleftrightarrow&\;\;\ell_A(x\apr y)c=\ell_A(x)\ell_A(y)c\;{\rm and}\;{\rm equation}\; Eq.~(\mref{eq:2.7})\;{\rm holds};\\
(a*y)*z=a*(y\apr z)\;&\Longleftrightarrow&\;\; r_A(y\apr
z)a=r_A(z)r_A(y)a \;{\rm and}\;{\rm equation}\; Eq.~(\mref{eq:2.6})\;{\rm holds}\\
(x*b)*z=x*(b*z) \;&\Longleftrightarrow&\;\; r_A(z)\ell_A(x)b
=\ell_A(x)r_A(z)b\;{\rm and}\;{\rm equation}\; Eq.~(\mref{eq:2.9})\;{\rm holds}\\
(a\rpr b)\rpr c=a\rpr (b\rpr c)\;&\Longleftrightarrow&\;\;B\;{\rm is}\;{\rm an}\;{\rm associative}\;{\rm algebra};\\
(a\rpr b)*z=a*(b*z)\;&\Longleftrightarrow&\;\;\ell_B(a\rpr b)z=\ell_B(a)\ell_B(b)z\;{\rm and}\;{\rm equation}\; Eq.~(\mref{eq:2.5})\;{\rm holds};\\
(x*b)*c=x*(b\rpr c)\;&\Longleftrightarrow&\;\; r_B(b\rpr
c)x=r_B(c)r_B(b)x \;{\rm and}\;{\rm equation}\; Eq.~(\mref{eq:2.4})\;{\rm holds}\\
(a*y)*c=a*(y*c) \;&\Longleftrightarrow&\;\; r_B(c)\ell_B(a)y
=\ell_B(a)r_B(c)y\;\text{ and Eq.~(\mref{eq:2.8}) holds}.
\end{eqnarray*}
Hence $A\bowtie B$ is a $\bfk$-algebra if and only if
$(\ell_A,r_A)$ is a bimodule of $A$ and $(\ell_B,r_B)$ is a bimodule of
$B$ and Eqs~(\mref{eq:2.4})-(\mref{eq:2.9}) hold. On the other hand,
if $A$ and $B$ are associative subalgebras of a $\bfk$-algebra
$C$ such that $C=A\oplus B$ as vector spaces, then it is easy to
show that the linear maps $\ell_A,r_A:A\rightarrow \End_\bfk(B)$ and
$\ell_B,r_B:B\rightarrow \End_\bfk(B)$ determined by
$$x*a=\ell_A(x)a+r_B(a)x,\;\;a*x=\ell_B(a)x+r_A(x)a,\;\;\forall x\in A,a\in B$$
satisfy equations Eqs.~(\mref{eq:2.4})-(\mref{eq:2.9}). In addition,
$(\ell_A,r_A)$ is a bimodule of $A$ and $(\ell_B,r_B)$ is a bimodule of
$B$.
 \end{proof}
} It is clear that an $A$-bimodule $\bfk$-algebra $(R,\ell,r)$ is
just a matched pair $(A,R, \ell,r,0,0)$. In fact, in this case,
Eq.~(\mref{eq:2.4}) is the second equation in
Eq.~(\mref{eq:twoalg1}), Eq.~(\mref{eq:2.5}) is the second
equation in Eq.~(\mref{eq:twoalg2}) and
Eq.~(\mref{eq:2.8}) is the second equation in
Eq.~(\mref{eq:twoalg3}). Equations (\mref{eq:2.6}),
(\mref{eq:2.7}) and (\mref{eq:2.9}) hold automatically.

In turn, an $A$-bimodule $(R,\ell,r)$ is the special case of an
$A$-bimodule $\bfk$-algebra when the multiplication
$\rpr$ on $R$ is the zero product. As a direct corollary of
Theorem~\mref{thm:mp}, we obtain the following result which is a
generalization of the classical result~\mcite{Sc} between bimodule
structures on $V$ and semi-direct product algebraic structures on
$A\oplus V$.

\begin{coro} If $(R,\rpr, \ell,r)$ is an $A$-bimodule $\bfk$-algebra, then the direct sum $A\oplus R$ of
vector spaces is turned into a $\bfk$-algebra (the semidirect
sum) by defining multiplication in $A\oplus R$ by
\begin{equation}
(x_1,v_1)\sempr(x_2,v_2)=(x_1\spr x_2,
\ell(x_1)v_2+v_1r(x_2)+v_1\rpr v_2),\;\;\forall x_1,x_2\in A,
v_1,v_2\in R.
\notag 
\end{equation}
\mlabel{co:twoalg}
\end{coro}
We denote this algebra by $A\ltimes_{\ell,r}R$ or simply $A\ltimes
R$.

\subsection{\Tto $\calo$-operators}
\mlabel{ss:moop} We first define an $\calo$-operator before
introducing an \tto $\calo$-operator through an auxiliary
operator.

\subsubsection{$\calo$-operators}

\begin{defn}{\rm Let $(A,\apr)$ be a $\bfk$-algebra and
$(R,\rpr, \ell,r)$ be an $A$-bimodule $\bfk$-algebra. A linear map
$\aop:R\rightarrow A$ is called an {\bf $\calo$-operator of weight
$\lambda\in {\bfk}$ associated to $(R,\rpr,\ell,r)$} if $\alpha$
satisfies
\begin{equation}
\aop(u)\apr
\aop(v)=\aop(\ell(\aop(u)v))+\aop(ur(\aop(v)))+\lambda\aop(u\rpr
v),\;\;\forall u,v\in V. \mlabel{eq:aop}
\end{equation}
} \notag 
\end{defn}

\begin{remark}{\rm Under our assumption that $\bfk$ is a field, the
non-zero weight can be normalized to weight 1. In fact, for a
non-zero weight $\lambda\in {\bfk}$, if $\alpha$ is an
$\calo$-operator of weight $\lambda$ associated to an $A$-bimodule
$\bfk$-algebra $(R,\rpr,\ell,r)$, then $\alpha$ is an $\calo$-operator of weight 1 associated to $(R,\lambda\rpr,\ell,r)$
and $\frac{1}{\lambda}\alpha$ is an $\calo$-operator of weight 1
associated to $(R,\rpr,\ell,r)$.}
\end{remark}

Note that, an $A$-bimodule $(V,\ell,r)$ becomes an $A$-bimodule
$\bfk$-algebra when $V$ is equipped with the zero multiplication.
Then a linear map $\aop:V\rightarrow A$ is an $\calo$-operator (of
any weight $\lambda$) if
\begin{equation}
\aop(u)\apr
\aop(v)=\aop(\ell(\aop(u))v)+\aop(ur(\aop(v))),\;\;\forall u,v\in V.
\notag 
\end{equation}
Such a structure appeared independently in \cite{Uc} under the name
of generalized Rota-Baxter operator. Since the weight $\lambda$ makes no difference in the definition, we
just call $V$ an $\calo$-operator. This definition recovers the
definition in Eq.~(\mref{eq:aop0}).

Obviously, an $\calo$-operator associated to $(A,L,R)$ is just a
Rota-Baxter operator on $A$. An $\calo$-operator can be viewed as
the relative version of a Rota-Baxter operator in the sense that the domain and range of an $\calo$-operator might be different.

\subsubsection{Balanced homomorphisms}
For our purpose of further generalizing the concept of an $\calo$-operator, we introduce another concept.
\begin{defn}
{\rm
Let $(A,\apr)$ be a $\bfk$-algebra.
\begin{enumerate}
\item
Let $\kappa\in\bfk$ and let $(V,\ell,r)$ be an $A$-bimodule. A linear map (resp. an $A$-bimodule homomorphism) $\beta:V\rightarrow A$ is
called a {\bf balanced linear map of \bwt $\kappa$} (resp. {\bf balanced $A$-bimodule homomorphism of \bwt $\kappa$}) if
\begin{equation}
\kappa \ell(\beta(u))v = \kappa ur(\beta(v)), \quad \forall\, u,v\in V.
\mlabel{eq:ksy}
\end{equation}
\mlabel{it:blm}
\item
Let $\kappa,\mu\in\bfk$ and let $(R, \rpr,\ell,r)$ be an $A$-bimodule $\bfk$-algebra. A linear map (resp. an $A$-bimodule homomorphism)
$\beta:R\rightarrow A$ is called a {\bf balanced linear map of \bwt
($\kappa,\mu$)} (resp. a {\bf balanced $A$-bimodule homomorphism of \bwt $(\kappa,\mu)$}) if Eq.~(\mref{eq:ksy}) holds and
\begin{equation}
\mu \ell(\beta(u\rpr v))w=\mu ur(\beta(v\rpr w)), \quad \forall\,
u,v,w\in R. \mlabel{eq:mueq}
\end{equation}
\end{enumerate}
}
\mlabel{de:co}
\end{defn}

Clearly, if $\kappa=0$ (resp. $\mu=0$), then Eq.~(\mref{eq:ksy})
(resp. Eq.~(\mref{eq:mueq})) imposes no restriction. So any
$A$-bimodule homomorphism is balanced of \bwt $(\kappa,\mu)=(0,0)$.
For a non-zero \bwt, we have the following
examples.

\begin{exam}{\rm Let $A$ be a $\bfk$-algebra.
\begin{enumerate}
\item The identity map $\beta={\rm id}:(A,L,R)\rightarrow A$ is a
balanced $A$-bimodule homomorphism (of any \bwt $(\kappa,\mu)$).
\item Any $A$-bimodule homomorphism $\beta: (A,L,R)\rightarrow A$ is balanced (of any \bwt $(\kappa,\mu)$).
\item Let $r\in A\otimes A$ be symmetric. If $r$ is regarded as a linear
map from $(A^*,R^*,L^*)$ to $A$ is an $A$-bimodule homomorphism,
then $r$ is a balanced $A$-bimodule homomorphism (of any \bwt
$\kappa$). See
Lemma~\mref{le:syin}.
\end{enumerate}}
\end{exam}

\subsubsection{\Tto $\calo$-operators}
We can now introduce our first main concept in this paper.
\begin{defn}
{\rm Let $(A,\apr)$ be a $\bfk$-algebra and let $(R,\rpr, \ell,r)$
be an $A$-bimodule $\bfk$-algebra.
\begin{enumerate}
\item Let $\lambda,\kappa,\mu\in \bfk$. Fix a balanced
$A$-bimodule homomorphism  $\beta:(R,\ell,r)\to A$ of \bwt
$(\kappa,\mu)$. A linear map $\alpha: R\to A$ is called an {\bf
\tto $\calo$-operator of weight $\lambda$ with \bop $\beta$ of
\bwt $(\kappa,\mu)$} if
\begin{equation}
\alpha(u)\apr\alpha(v)-\alpha(\ell(\alpha(u))v+ur(\alpha(v))+\lambda
u\rpr v)=\kappa \beta(u)\apr\beta(v)+\mu\beta(u\rpr v), \quad \forall\, u,v\in R.
\mlabel{eq:gmybe}
\end{equation}
\item We also let $(\alpha,\beta)$ denote an \tto $\calo$-operator
$\alpha$ with \bop $\beta$. \item When $(V,\ell,r)$ is an
$A$-bimodule, we regard $V$ as an $A$-bimodule $\bfk$-algebra with
the zero multiplication. Then $\lambda$ and $\mu$ are irrelevant.
We then call the pair $(\alpha,\beta)$ an {\bf \tto $\calo$-operator
with \bop $\beta$ of \bwt $\kappa$.}
\end{enumerate}
}
\mlabel{de:mop}
\end{defn}
We note that, when the \bop $\beta$ is the zero map (and hence
$\kappa$ and $\mu$ are irrelevant), then $\alpha$ is the
$\calo$-operator defined in Definition 2.7.

\subsection{\Tto $\calo$-operators and associativity}
\mlabel{ss:oass}
Let $(A,\apr)$ be a $\bfk$-algebra and $(R,\rpr, \ell,r)$ be an
$A$-bimodule $\bfk$-algebra. Let $\delta_{\pm}:R\rightarrow A$ be
two linear maps and $\lambda\in\textbf{k}$. We now consider the associativity of the multiplication
\begin{equation}
u\dpr v:=\ell(\delta_{+}(u))v+ur(\delta_{-}(v))+\lambda u\rpr v,\quad
\forall\, u,v\in R,\mlabel{eq:product}
\end{equation}
and several other related multiplications. This will be applied in the Section~\mref{sec:maybe}.

Let the characteristic of the field $\bfk$ be different from 2.  Set
\begin{equation}
\alpha:=(\delta_{+}+\delta_{-})/2,\quad
\beta:=(\delta_{+}-\delta_{-})/2,\mlabel{eq:alphabeta-1}
\end{equation}
called the {\bf symmetrizer} and {\bf antisymmetrizer} of
$\delta_\pm$ respectively. Note that $\delta_\pm$ can be recovered
from $\alpha$ and $\beta$ by $\delta_{\pm}=\alpha\pm\beta$.

\begin{lemma} Let $(A,\apr)$ be a $\bfk$-algebra and
$(R,\rpr, \ell,r)$ be an $A$-bimodule $\bfk$-algebra. Let
$\alpha:R\rightarrow A$ be a linear map and let
$\lambda\in\textbf{k}$. Then the operation given by
\begin{equation}
u*_\alpha v:=\ell(\alpha(u))v+ur(\alpha(v))+\lambda u\rpr v,\;\;\forall u,v\in
R, \end{equation}
is associative if
and only if
\begin{equation}
\ell\big(\alpha(u)\apr\alpha(v)-\alpha(u\ast_\alpha v)\big)w=ur\big(\alpha(v)\apr\alpha(w)-\alpha(v\ast_\alpha w)\big),\mlabel{eq:condition}
\end{equation}
for all $u,v,w\in R$.\mlabel{le:product}
\end{lemma}

\begin{proof}
For any $u,v,w\in R$, we have
\begin{eqnarray}
(u\ast_\alpha v)\ast_\alpha w&=& \ell(\alpha(u\ast_\alpha v))w + (u\ast_\alpha v)r(\alpha(w)) + \lambda (u\ast_\alpha v) \rpr w
\notag \\
&=&\ell(\alpha(u\ast_\alpha v))w +(\ell(\alpha(u))v)r(\alpha(w))+(ur(\alpha(v)))r(\alpha(w))
\mlabel{eq:assl}\\
& &+\lambda(u\rpr v)r(\alpha(w))+\lambda(\ell(\alpha(u))v)\rpr
w+\lambda(ur(\alpha(v)))\rpr w+\lambda^2(u\rpr v)\rpr w.
\notag
\end{eqnarray}
and
\begin{eqnarray}
u\ast_\alpha(v\ast_\alpha w)&=& \ell(\alpha(u))(v\ast_\alpha w) + u r(\alpha(v\ast_\alpha w)) + \lambda u\rpr (v\ast_\alpha w)\notag\\
&=&\ell(\alpha(u))(\ell(\alpha(v))w)+\ell(\alpha(u))(vr(\alpha(w)))+\lambda
\ell(\alpha(u))(v\rpr w)\mlabel{eq:assr}\\
&& +u r(\alpha(v\ast_\alpha w))+\lambda
u\rpr(\ell(\alpha(v))w)+\lambda
u\rpr(vr(\alpha(w)))+\lambda^2u\rpr(v\rpr w).\notag
\end{eqnarray}
Since $(R,\rpr,\ell,r)$ is an $A$-bimodule $\bfk$-algebra, the second,
fourth, fifth, sixth and seventh term on the right hand side of
Eq.~(\mref{eq:assl}) agrees with the second, sixth, third, fifth and
seventh term on the right hand side of Eq.~(\mref{eq:assr})
respectively. We further have
$(ur(\alpha(v))r(\alpha(w))=ur(\alpha(v)\rpr \alpha(w))$ and
$\ell(\alpha(u))(\ell(\alpha(v)) w)=\ell(\alpha(u)\rpr \alpha(v))w$.
We thus have
$$(u\ast_\alpha
v)\ast_\alpha w-u\ast_\alpha(v\ast_\alpha w)=
ur\big(\alpha(v)\apr\alpha(w)-\alpha(v\ast_\alpha w)\big)
- \ell\big(\alpha(u)\apr\alpha(v)-\alpha(u\ast_\alpha v)\big)w.$$
Then the lemma follows.
\end{proof}

\begin{coro} Let $\bfk$ be a field of characteristic not equal to 2. Let $(A,\apr)$ be a $\bfk$-algebra and
$(R,\rpr, \ell,r)$ be an $A$-bimodule $\bfk$-algebra. Let
$\delta_{\pm}:R\rightarrow A$ be two linear maps and
$\lambda\in\textbf{k}$. Let $\alpha$ and $\beta$ be their
symmetrizer and antisymmetrizer defined by
Eq.~$($\mref{eq:alphabeta-1}$)$. If $\beta$ is a balanced linear map of
\bwt $\kappa=1$, namely
\begin{equation}
\ell(\beta(u))v=ur(\beta(v)),\;\; \forall u,v\in R, \mlabel{eq:sy}
\end{equation}
then the operation $\dpr$ in
Eq.~$($\mref{eq:product}$)$ defines an associative product on $R$ if and
only if $\alpha$ satisfies
Eq.~$($\mref{eq:condition}$)$.\mlabel{co:product}
\end{coro}

\begin{proof}
Since $\beta$ is balanced, for any $u,v\in R$
$$u\dpr v=\ell(\delta_{+}(u))v+ur(\delta_{-}(v))+\lambda u\rpr
v=\ell(\alpha(u))v+ur(\alpha(v))+\lambda u\rpr v.$$ Then the
conclusion follows from Lemma~\mref{le:product}.
\end{proof}

Obviously, if $\alpha$ is an $\calo$-operator of weight $\lambda$
associated to an $A$-bimodule $\bfk$-algebra $(R,\rpr, \ell,r)$, then
Eq.~(\mref{eq:condition}) holds. Thus the operation on $R$ defined by
Eq.~(\mref{eq:product}) is associative.

\begin{theorem}
Let $\bfk$ have characteristic not equal to 2. Let $(A,\apr)$ be a $\bfk$-algebra and $(R,\rpr, \ell,r)$ be an
$A$-bimodule $\bfk$-algebra. Let $\delta_{\pm}:R\rightarrow A$ be
two linear maps and $\lambda\in\textbf{k}$. Let $\alpha$ and $\beta$
be the symmetrizer and antisymmetrizer of $\delta_\pm$.
\begin{enumerate}
\item
Suppose that $\beta$ is a balanced linear map of \bwt $(\kappa,\mu)$
and $\alpha$ satisfies Eq.~$($\mref{eq:gmybe}$)$. Then the product
$\ast_\alpha$ is associative. \mlabel{it:an1}
\item
Suppose $\beta$ is a balanced $A$-bimodule homomorphism of \bwt
$(-1,\pm \lambda)$, that is, $\beta$ satisfies Eq.~$($\mref{eq:ksy}$)$
with $\kappa=-1$, Eq.~$($\mref{eq:mueq}$)$ with $\mu=\pm\lambda$ and
\begin{equation}
\beta(\ell(x)u)=x\apr\beta(u),\quad \beta(ur(x))=\beta(u)\apr x,\quad
\forall x\in A,u\in R.\mlabel{eq:bimoho}
\end{equation}
Then $\alpha$ is an \tto $\calo$-operator of weight $\lambda$ with
\bop $\beta$ of \bwt $(\kappa,\mu)=(-1,\pm \lambda)$ if and only
if $\delta_{\pm}$ is an $\calo$-operator of weight 1 associated to
a new $A$-bimodule $\bfk$-algebra $(R,\rprpm_{\pm}, \ell,r)$:
\begin{equation}
\delta_{\pm}(u)\apr\delta_{\pm}(v) =\delta_{\pm}(\ell(\delta_{\pm}(u))v+ur(\delta_{\pm}(v)) +u\rprpm_{\pm}v),\quad
\forall u,v\in R,
\mlabel{eq:newo}
\end{equation}
where the associative products $\rprpm_{\pm}=\rprpm_{\lambda,\beta,\pm}$ on $R$ are
defined by
\begin{equation}
u\rprpm_{\pm}v=\lambda u\rpr v\mp2\ell(\beta(u))v,\quad \forall
u,v\in R.\mlabel{eq:diamond}
\end{equation}
\mlabel{it:an2}
\end{enumerate}
\mlabel{thm:ansatz}
\end{theorem}
Note that in Item~(\mref{it:an1}) we do not assume that $\beta$ is
an $A$-bimodule homomorphism. Thus $\alpha$ needs not be an \tto
$\calo$-operator.

\begin{proof}
(\mref{it:an1}) By the relations (\mref{eq:ksy}) and
(\mref{eq:mueq}) and the fact that $(R,\rpr, \ell,r)$ is an
$A$-bimodule $\bfk$-algebra, we obtain
$$\ell(\kappa \beta(u)\apr\beta(v)+\mu\beta(u\rpr
v))w=ur(\kappa \beta(v)\apr\beta(w)+\mu\beta(v\rpr w)),\;\;\forall
u,v,w\in R.$$ Since Eq.~(\mref{eq:gmybe}) also holds, the above
equation implies Eq.~(\mref{eq:condition}) and hence the associativity of $\ast_\alpha$ by Lemma~\mref{le:product}.
\smallskip

\noindent (\mref{it:an2}) First we prove that the operations
$\rprpm_{\pm}$ defined by Eq.~(\mref{eq:diamond}) make
$(R,\rprpm_\pm, \ell,r)$ into an $A$-bimodule $\bfk$-algebra. In fact,
for any $u,v,w\in R$,
{\allowdisplaybreaks
\begin{eqnarray*}
(u\rprpm_{\pm}v)\rprpm_{\pm}w&=&(\lambda u\rpr
v\mp2l(\beta(u))v)\rprpm_{\pm}w\\
&=&\lambda^2(u\rpr v)\rpr w\mp2\lambda(\ell(\beta(u))v)\rpr
w\mp2\lambda \ell(\beta(u\rpr v))w+4\ell(\beta(\ell(\beta(u))v))w\\
&=&\lambda^2u\rpr(v\rpr w)\mp2\lambda
u\rpr(\ell(\beta(v))w)\mp2\lambda \ell(\beta(u))(v\rpr
w)+4\ell(\beta(u))(\ell(\beta(v))w)\\
&=&u\rprpm_{\pm}(\lambda v\rpr
w\mp2\ell(\beta(v))w)=u\rprpm_{\pm}(v\rprpm_{\pm}w),
\end{eqnarray*}
} where the third equality follows since each term on one side of
the equation equals to the corresponding term on the other side by
our assumptions on $\beta$ and the fact that $(R,\rpr, \ell,r)$ is an
$A$-bimodule $\bfk$-algebra. On the other hand, for any $x\in A$,
{\allowdisplaybreaks
\begin{eqnarray*}
\ell(x)(u\rprpm_{\pm}v)&=&\ell(x)(\lambda u\rpr v\mp2\ell(\beta(u))v)\\
&=&\lambda(\ell(x)u)\rpr v\mp 2\ell(x\apr \beta(u))v\quad ({\rm
by\;Eq.~(\mref{eq:twoalg1})})\\
&=&\lambda(\ell(x)u)\rpr v\mp2\ell(\beta(\ell(x)u))v\quad ({\rm
by\;Eq.~(\mref{eq:bimoho})})\\
&=&(\ell(x)u)\rprpm_{\pm}v.
\end{eqnarray*}
}
By the same argument, we have
$(u\rprpm_{\pm}v)r(x)=u\rprpm_{\pm}(vr(x)).$
Moreover,
\begin{eqnarray*}
(ur(x))\rprpm_{\pm}v&=&\lambda(ur(x))\rpr v\mp2\ell(\beta(ur(x)))v\\
&=&\lambda u\rpr(\ell(x)v)\mp2\ell(\beta(u))(\ell(x)v)\quad {\rm
(by \; Eq.~(\mref{eq:bimoho}),\;
Eq.~(\mref{eq:twoalg1})\; and\;
Eq.~(\mref{eq:twoalg3}))}\\
&=&u\rprpm_{\pm}(\ell(x)v).
\end{eqnarray*}

The other axioms in the Definition~\mref{de:bimal} of an
$A$-bimodule $\bfk$-algebra do not depend on the product of $R$.
Thus $(R,\rprpm_\pm, \ell,r)$ equipped with the product
$\rprpm_{\pm}$ is an $A$-bimodule $\bfk$-algebra. Moreover,
{\allowdisplaybreaks
\begin{eqnarray*}
&
&(\alpha\pm\beta)(u)\apr(\alpha\pm\beta)(v)-
(\alpha\pm\beta)(\ell((\alpha\pm\beta)(u))v+ur((\alpha\pm\beta)(v))+u\rprpm_{\pm}v)\\
&=&\alpha(u)\apr\alpha(v)+\beta(u)\apr\beta(v) -\alpha(\ell(\alpha(u))v+ur(\alpha(v))+\lambda
u\rpr v)\mp\lambda\beta(u\rpr v)\\
&& \pm(\beta(u)\apr\alpha(v)-\beta(ur(\alpha(v))) +\alpha(u)\apr\beta(v)-\beta(\ell(\alpha(u))v))\quad
({\rm by\; Eq.~(\mref{eq:sy})})\\
&=&\alpha(u)\apr\alpha(v)+\beta(u)\apr\beta(v)-\alpha(\ell(\alpha(u))v+ur(\alpha(v))+\lambda
u\rpr v)\mp\lambda\beta(u\rpr v)\ ({\rm by\;
Eq.~(\mref{eq:bimoho})}).
\end{eqnarray*}
}
So $\alpha$ and $\beta$ satisfy Eq.~(\mref{eq:gmybe}) with $\kappa =-1$ and
$\mu=\pm\lambda$ if and only if $\delta_{\pm}$ is an
$\calo$-operator of weight 1 from $(R,\rprpm_\pm, \ell,r)$ to $A$.
\end{proof}

We close this section with an obvious corollary of
Theorem~\mref{thm:ansatz}.

\begin{coro}
Let $A$ be a $\bfk$-algebra and $(V,\ell,r)$ be an $A$-bimodule.
Let $\alpha,\beta:V\rightarrow A$ be two linear maps such that
$\beta$ is a balanced $A$-bimodule homomorphism. Then $\alpha$ is
an \tto $\calo$-operator with \bop $\beta$ of \bwt $\kappa=-1$ if
and only if $\alpha\pm\beta$ is an $\calo$-operator of weight 1
associated to an $A$-bimodule $\bfk$-algebra
$(V,\star_{\pm},\ell,r)$, i.e.,
$$
(\alpha\pm\beta)(u)\apr(\alpha\pm\beta)(v)=(\alpha\pm\beta)(\ell((\alpha\pm\beta)(u))v+
ur((\alpha\pm\beta)(v))+u\star_{\pm}v),\quad \forall u,v\in V,
$$
where the associative algebra products $\star_{\pm}$ on $V$ are
defined by
$$
u\star_{\pm}v=\mp2\ell(\beta(u))v,\quad \forall u,v\in V.
$$
\mlabel{co:an1}
\end{coro}
\begin{proof}
The corollary follows by taking $R=V$ in Theorem~\mref{thm:ansatz}.
(\mref{it:an2}) with the zero multiplication.
\end{proof}

\subsection{The case of $\calo$-operators and Rota-Baxter operators}
\mlabel{ss:rb}

Let $(A,\apr)$ be a $\bfk$-algebra. Then $(A,\cdot, L,R)$ is an
$A$-bimodule $\bfk$-algebra. Theorem~\mref{thm:ansatz} can be easily
restated in this case. But we are mostly interested
in the case of $\mu=0$ when Eq.~(\mref{eq:gmybe}) takes the form
\begin{equation}
\aop(x)\apr \aop(y)-\aop(\aop(x)\apr y+x\apr \aop(y)+\lambda
x\apr y)=\kappa \beta(x)\apr\beta(y),\quad \forall x,y\in
A.\mlabel{eq:lambdak}
\end{equation}
We list the following special cases for later reference. When $\lambda=0$,
Eq.~(\mref{eq:lambdak}) gives
\begin{equation}
\aop(x)\apr \aop(y)-\aop(\aop(x)\apr y+x\apr \aop(y))=\kappa
\beta(x)\apr\beta(y),\quad \forall x,y\in A.\mlabel{eq:alphak}
\end{equation}
If in addition, $\beta={\rm id}$, then Eq.~(\mref{eq:alphak}) gives
\begin{equation}
\aop(x)\cdot\aop(y)-\aop (\aop(x)\cdot y+x\cdot\aop(y))=\kappa x\cdot
y,\;\;\forall x,y\in A.
\mlabel{eq:kmyb}
\end{equation}
When furthermore $\kappa =-1$, Eq.~(\mref{eq:kmyb}) becomes
\begin{equation}
\aop(x)\cdot\aop(y)-\aop (\aop(x)\cdot y+x\cdot\aop(y))=-x\cdot
y,\;\;\forall x,y\in A. \mlabel{eq:-1myb}
\end{equation}

By the proof of Lemma~\mref{le:product} and Theorem~\mref{thm:ansatz}, we reach the following conclusion.
\begin{coro}
Let $(A,\apr)$ be a $\bfk$-algebra.  Let
$\alpha,\beta:A\rightarrow A$ be two linear maps and
$\lambda\in\textbf{k}$.
\begin{enumerate}
\item For any $\kappa\in\textbf{k}$, let $\beta$ be balanced of
\bwt $(\kappa,0)$ and let $\alpha$ be an \tto $\calo$-operator of
weight $\lambda$ with \bop $\beta$ of \bwt
$(\kappa,\mu)=(\kappa,0)$, namely $\alpha$ satisfies
Eq.~$($\mref{eq:lambdak}$)$. Then the product $\ast_\alpha$ on $A$ is
associative.

\item
If $\beta$ is an $A$-bimodule homomorphism, then $\alpha$ and $\beta$ satisfy
Eq.~$($\mref{eq:alphak}$)$ for $\kappa=-1$ if
and only if $r_{\pm}=\alpha\pm\beta$ is an $\calo$-operator of
weight 1 associated to a new $A$-bimodule $\bfk$-algebra
$(A,\star_{\pm}, L,R)$:
$$
r_{\pm}(x)\apr r_{\pm}(y)=r_{\pm}(r_{\pm}(x)\apr y+x\apr
r_{\pm}(y)+x\star_{\pm} y),\quad \forall x,y\in A,
$$
where the associative products $\star_{\pm}$ on $A$ are
defined by
\begin{equation}
x\star_{\pm}y=\mp2\beta(x)\apr y,\quad \forall x,y\in A.
\notag
\end{equation}
\end{enumerate}
\mlabel{co:aasso}
\end{coro}

\begin{remark}
{\rm  With the notations as above, if $\beta:A\rightarrow A$ is an
$A$-bimodule homomorphism, then $\beta$ is balanced of \bwt
$(\kappa, 0)$. Moreover, in this case, $\beta$ is an {\bf averaging
operator}~\mcite{Ag2,R3}, namely,
\begin{equation}
\beta(x)\apr\beta(y)=\beta(x\apr\beta(y))=\beta(\beta(x)\apr
y),\quad \forall x,y\in A,
\end{equation}
and it is also a {\bf Nijenhuis operator}~\mcite{CGM,E2}, namely,
\begin{equation}
\beta(x)\apr\beta(y)+\beta^2(x\apr
y)=\beta(x\apr\beta(y)+\beta(x)\apr y),\quad \forall x,y\in
A.\notag
\end{equation}}
\end{remark}

Let $(A,\apr)$ be a $\bfk$-algebra and let $(A,\cdot, L,R)$ be the
corresponding $A$-bimodule $\bfk$-algebra. In this case, it is
obvious that $\beta={\rm id}$ satisfies the conditions of
Theorem~\mref{thm:ansatz} and Eq.~(\mref{eq:gmybe}) takes the form
\begin{equation}
\aop(x)\apr \aop(y)-\aop(\aop(x)\apr y+x\apr \aop(y)+\lambda
x\apr y)=\hat{\kappa }x\apr y,\quad \forall x,y\in A,\mlabel{eq:pgmybe}
\end{equation}
where $\hat{\kappa }=\kappa +\mu$. Thus we have the following consequence of Theorem~\mref{thm:ansatz}.
\begin{coro}
Let $\hat{\kappa }=-1\pm\lambda$. Then $\aop:A\rightarrow A$ satisfies
Eq.~$($\mref{eq:pgmybe}$)$ if and only if $\aop\pm1$ is a Rota-Baxter
operator of weight $\lambda\mp2$.
\mlabel{co:mop}
\end{coro}
When $\lambda=0$, this fact can be found in~\mcite{E1}. As noted
there, the Lie algebraic version of Eq.~(\mref{eq:pgmybe}) in this case, namely Eq.~(\mref{eq:-1myb}), is
the operator form of the modified classical Yang-Baxter
equation~\mcite{Se}.

\section{\Tto $\calo$-operators and AYBE}
\mlabel{sec:aybe} In this section we study the relationship
between \tto $\calo$-operators and associative Yang-Baxter
equations. We start with introducing various concepts of the
associative Yang-Baxter equation (AYBE) in Section~\mref{ss:aybe}.
We then establish connections between $\calo$-operators in
different generalities and solutions of these variations of AYBE
in different algebras. The relationship between $\calo$-operators
on a $\bfk$-algebra $A$ and solutions of AYBE in $A$ is considered
in Section~\mref{ss:oaybe}. We consider the special case of
Frobenius algebras in Section~\mref{ss:frob}. We finally consider
in Section~\mref{ss:goaybe} the relationship between an \tto
$\calo$-operator and solutions of AYBE and \tte AYBE in an
extension algebra of $A$.

\subsection{\Tto associative Yang-Baxter equations}
\mlabel{ss:aybe}

We define variations of the associative Yang-Baxter equation to be satisfied by two tensors from an algebra. We then study the linear maps from these two tensors in preparation for the relationship between $\calo$-operators and solutions of these associative Yang-Baxter equations.

Let $A$ be a $\bfk$-algebra.
Let $r=\sum_i a_i\ot b_i\in A\ot A$. We continue to use the notations $r_{12}, r_{13}$ and $r_{23}$ defined in Eq.~(\mref{eq:r12}).
We similarly define
\begin{equation}
r_{21}=\sum_{i} b_i\ot a_i\ot 1, \quad
r_{31}=\sum_{i}b_i\ot 1 \ot a_i, \quad r_{32}=\sum_{i}
1\ot b_i\ot a_i. \notag
\end{equation}

Equip $A\ot A\ot A$ with the product of the tensor algebra. In particular,
\begin{equation}
 (a_1\ot a_2\ot a_3) (b_1\ot b_2\ot b_3) = (a_1 b_1)\ot (a_2 b_2) \ot (a_3 b_3), \quad \forall\, a_i,b_i\in A, i=1,2,3.
\notag 
\end{equation}

\begin{defn}
{\rm Fix $\esym\in \bfk$.
\begin{enumerate}
\item
The equation
\begin{equation}
r_{12} r_{13}+r_{13} r_{23}-r_{23} r_{12} =\vep (r_{13}+r_{31}) (r_{23}+r_{32})
\mlabel{eq:type2aybe}
\end{equation}
is called the {\bf \tte associative Yang-Baxter equation of \ewt $\esym$} (or {\bf $\esym$-\MAYBE} in short).
\item
Let $A$ be a $\bfk$-algebra. An element $r\in A\otimes A$ is called a {\bf solution of the $\esym$-\MAYBE in $A$} if it satisfies Eq.~(\mref{eq:type2aybe}).
\end{enumerate}
}
\mlabel{de:ayb2}
\end{defn}

When $\esym=0$ or $r$ is skew-symmetric in the sense that $\sigma(r)=-r$ for the switch operator $\sigma: A\ot A\to A\ot A$ (and hence $r_{13}=-r_{31}$), then the $\esym$-\MAYBE is
the same as the AYBE in Eq.~(\mref{eq:aybe}):
\begin{equation}
r_{12} r_{13}+r_{13} r_{23}-r_{23} r_{12}=0.
\mlabel{eq:aybe2}
\end{equation}

Let $A$ be a $\bfk$-algebra with finite $\bfk$-dimension. For $r\in
A\otimes A$, define a linear map $F_r:A^*\to A$ by \begin{equation}
\langle v,F_r(u)\rangle =\langle u\otimes v,r\rangle,\;\;\forall
u,v\in A^*.
\end{equation}
This defines a bijective linear map $F:A\ot A\to \Hom_\bfk(A^*,A)$ and thus allows us to identify $r$ with $F_r$ which we still denote by $r$ for simplicity of notations.
Similarly define a linear map $r^t:A^*\rightarrow A$ by
\begin{equation}
\langle u,r^t(v)\rangle =\langle r,u\otimes v\rangle.
\end{equation}
Obviously $r$ is symmetric or skew-symmetric in $A\ot A$ if and only if, as a linear map, $r=r^t$
or $r=-r^t$ respectively.
Suppose that the characteristic of $\bfk$ is not 2 and define
\begin{equation}
\alpha=\alpha_r=(r-r^t)/2,\quad \beta=\beta_r=(r+r^t)/2,\mlabel{eq:alphabeta}
\end{equation}
called the {\bf skew-symmetric part} and
the {\bf symmetric part} of $r$ respectively. Then $r=\alpha+\beta$
and $r^t=-\alpha+\beta$.

\begin{lemma}
Let $(A,\apr)$ be a $\bfk$-algebra with finite $\bfk$-dimension. Let $s\in A\otimes A$
be symmetric. Then the following conditions are equivalent.
\begin{enumerate}
\item
$s$ is {\bf invariant}, i.e.,
\begin{equation}
({\rm id}\otimes L(x)-R(x)\otimes{\rm id})s=0,\;\;\forall x\in
A.\mlabel{eq:invariant}
\end{equation}\mlabel{it:sin}
\item
 $s$ regarded as a linear map from
$(A^*,R^*,L^*)$ to $A$ is balanced, i.e.,
\begin{equation}
R^*(s(a^*))b^*=a^*L^*(s(b^*)),\quad \forall a^*,b^*\in
A^*.\mlabel{eq:dualssy}
\end{equation}\mlabel{it:ssy}
\item
$s$ regarded as a linear map from $(A^*,R^*,L^*)$ to $A$ is an
$A$-bimodule homomorphism, i.e.,
\begin{equation}
s(R^*(x)a^*)=x\apr s(a^*),\quad s(a^*L^*(x))=s(a^*)\apr x,\quad
\forall x\in A,a^*\in A^*.\mlabel{eq:dualsabi}
\end{equation}\mlabel{it:sbi}
\end{enumerate} \mlabel{le:syin}
\end{lemma}

\begin{proof}
``(\mref{it:sin})$\Leftrightarrow$(\mref{it:ssy})". Since $s\in A\otimes A$ is symmetric, for any $x\in
A, a^*,b^*\in A^*$,
\begin{eqnarray*}
\langle({\rm id}\otimes L(x)-R(x)\otimes{\rm id})s,a^*\otimes
b^*\rangle&=&\langle s,a^*\otimes L^*(x)b^*\rangle-\langle
s,R^*(x)a^*\otimes b^*\rangle\\
&=&\langle x\apr s(a^*),b^*\rangle-\langle a^*,s(b^*)\apr
x\rangle\\
&=&\langle R^*(s(a^*))b^*-a^*L^*(s(b^*)),x\rangle.
\end{eqnarray*}
So $s$ is invariant if and only if $s$ regarded as a linear map from
$(A^*,R^*,L^*)$ to $A$ is balanced.

``(\mref{it:sin})$\Leftrightarrow$(\mref{it:sbi})". For any $x\in
A,a^*,b^*\in A^*$,
\begin{eqnarray*}
\langle({\rm id}\otimes L(x)-R(x)\otimes{\rm id})s,a^*\otimes
b^*\rangle&=&\langle s,a^*\otimes L^*(x)b^*\rangle-\langle
s,R^*(x)a^*\otimes b^*\rangle\\
&=&\langle x\apr s(a^*)-s(R^*(x)a^*),b^*\rangle\\
\langle({\rm id}\otimes L(x)-R(x)\otimes{\rm id})s,a^*\otimes
b^*\rangle&=&\langle s,a^*\otimes L^*(x)b^*\rangle-\langle
s,R^*(x)a^*\otimes b^*\rangle\\
&=&\langle s(L^*(x)b^*)-s(b^*)\apr x,a^*\rangle,
\end{eqnarray*}
by the symmetry of $s\in A\otimes A$. So $s$ is invariant if and
only if $s$ regarded as a linear map from $(A^*,R^*,L^*)$ to $A$ is
an $A$-bimodule homomorphism.
\end{proof}

\begin{remark}
{\rm The invariant condition in Item~(\mref{it:sin}) also arises in the
construction of a coboundary antisymmetric infinitesimal bialgebra
in the sense of~\mcite{Bai2} (see also~\mcite{NBG}). }
\end{remark}

\subsection{\Tto $\calo$-operators and \MAYBE}
\mlabel{ss:oaybe}
We first state the following special case of Corollary~\mref{co:an1}.
\begin{coro}
Let $\bfk$ be a field of characteristic not equal to 2. Let $A$ be a $\bfk$-algebra with finite $\bfk$-dimension and $r\in A\otimes A$.
Let $\alpha$, $\beta$ be defined by Eq.~$($\mref{eq:alphabeta}$)$. Suppose $\beta$ is a balanced $A$-bimodule homomorphism. The following two statements are equivalent.
\begin{enumerate}
\item The map $\alpha$ is an \tto $\calo$-operator with \bop
$\beta$ of \bwt $-1$:
\begin{equation}
\alpha(a^*)\apr
\alpha(b^*)-\alpha(R^*(\alpha(a^*))b^*+a^*L^*(\alpha(b^*)))=-\beta(a^*)\apr\beta(b^*),\quad
\forall a^*,b^*\in A^*.\mlabel{eq:-1maybe}
\end{equation}
\item
The map
 $r$ (resp. $-r^t$) is an $\calo$-operator of weight 1
associated to a new $A$-bimodule $\bfk$-algebra $(A^*,\rpr_+,R^*,L^*)$
$($resp. $(A^*,\rpr_-,R^*,L^*)$ $)$:
\begin{equation}
r(a^*)\apr r(b^*)=r(R^*(r(a^*))b^*+a^*L^*(r(b^*))+a^*\rpr_{+}
b^*),\quad \forall a^*,b^*\in A^*,\mlabel{eq:opweight1}
\end{equation}
$($resp.
\begin{equation}
(-r^t)(a^*)\apr(-r^t)(b^*)=(-r^t)(R^*((-r^t)(a^*))b^*+a^*L^*((-r^t)(b^*))+a^*\rpr_{-}
b^*), \mlabel{eq:topweight1}
\end{equation}
$\forall a^*,b^*\in A^*)$, where the associative algebra products $\rpr_{\pm}$ on $A^*$ are
defined by
\begin{equation}
a^*\rpr_{\pm} b^*=\mp2R^*(\beta(a^*))b^*,\quad \forall a^*,b^*\in
A^*.\mlabel{eq:betaproduct}
\end{equation}
\end{enumerate}
\mlabel{co:abas}
\end{coro}

In the theory of integrable systems~\mcite{Ko,Se}, {\bf modified
classical Yang-Baxter equation} is usually referred to (the Lie
algebraic version of) Eq.~(\mref{eq:-1myb}) and
Eq.~(\mref{eq:-1maybe}).

The following theorem establishes a close relationship between
\tto $\calo$-operators on a $\bfk$-algebra $A$ and solutions of the AYBE in
$A$.
\begin{theorem}
Let $\bfk$ be a field of characteristic not equal to 2. Let $A$ be a $\bfk$-algebra with finite $\bfk$-dimension and let $r\in A\otimes A$
which is identified as a linear map from $A^*$ to $A$.
\begin{enumerate}
\item
Then $r$ is a solution of the AYBE in $A$ if and only
if $r$ satisfies
\begin{equation}
r(a^*)\apr r(b^*)=r(R^*(r(a^*))b^*-a^*L^*(r^t(b^*))),\;\;\forall
a^*,b^*\in A^*.\mlabel{eq:aybeform}
\end{equation}
\mlabel{it:aybr}
\item
Define
$\alpha$ and $\beta$ by Eq.~$($\mref{eq:alphabeta}$)$. Suppose that the
symmetric part $\beta$ of $r$ is invariant. Then $r$ is a solution
of \MAYBE of \ewt $\frac{\kappa+1}{4}$:
\begin{equation}
r_{12} r_{13}+r_{13} r_{23}-r_{23} r_{12} =\frac{\kappa+1}{4}(r_{13}+r_{31}) (r_{23}+r_{32})
\notag 
\end{equation}
if and only if $\alpha$ is an \tto $\calo$-operator with \bop
$\beta$ of \bwt $\kappa$:
\begin{equation}
\alpha(a^*)\apr
\alpha(b^*)-\alpha(R^*(\alpha(a^*))b^*+a^*L^*(\alpha(b^*))) =\kappa\beta(a^*)\apr\beta(b^*),\quad
\forall a^*,b^*\in A^*.
\notag 
\end{equation}
\mlabel{it:aybea}
\end{enumerate}
\mlabel{thm:aybea}
\end{theorem}

\begin{proof}
(\mref{it:aybr})
Denote $r=\sum_{i,j}u_i\otimes v_j$. For any $a^*,b^*,c^*\in A^*$, we
have
\begin{eqnarray*}
\langle r_{12}\apr r_{13},a^*\otimes b^*\otimes c^*\rangle
&=&\sum_{i,j}\langle u_i\apr u_j,a^*\rangle \langle v_i,b^*\rangle
\langle v_j,c^*\rangle=\sum_j\langle
r^t(b^*)\apr u_j,a^*\rangle\langle v_j,b^*\rangle\\
&&=\langle r(a^*L^*(r^t(b^*))),c^*\rangle ,\notag \\
\langle r_{13}\apr r_{23},a^*\otimes b^*\otimes c^*\rangle
&=&\sum_{i,j}\langle u_i,a^*\rangle \langle u_j,b^*\rangle \langle
v_i\apr v_j,c^*\rangle=\sum_j\langle u_j,b^*\rangle\langle
r(a^*)\apr v_j,c^*\rangle\\
&=&\langle r(a^*)\apr r(b^*),c^*\rangle ,\\
\langle -r_{23}\apr r_{12},a^*\otimes b^*\otimes c^*\rangle
&=&-\sum_{i,j}\langle u_i,a^*\rangle \langle u_j\apr v_i,b^*\rangle
\langle v_j,c^*\rangle=-\sum_j\langle u_j\apr r(a^*),b^*\rangle
\langle v_j,c^*\rangle \\ &=&\langle -r(R^*(r(a^*))b^*),c^*\rangle
.\notag
\end{eqnarray*}
Therefore $r$ is a solution of the AYBE in
$A$ if and only if $r$ satisfies Eq.~(\mref{eq:aybeform}).
\smallskip

\noindent
(\mref{it:aybea})
By the proof of Item~(\mref{it:aybr}), we see that, for
any $a^*,b^*,c^*\in A^*$,
{\allowdisplaybreaks
\begin{eqnarray*}
& &\langle \alpha(a^*)\apr
\alpha(b^*)-\alpha(R^*(\alpha(a^*))b^*+a^*L^*(\alpha(b^*))) -\kappa\beta(a^*)\apr\beta(b^*),c^*\rangle\\
&=&\langle \alpha(a^*)\apr
\alpha(b^*)-\alpha(R^*(\alpha(a^*))b^*+a^*L^*(\alpha(b^*))) +\beta(a^*)\apr\beta(b^*)-
(\kappa+1)\beta(a^*)\apr\beta(b^*),c^*\rangle\\
&=&\langle r_{12}\apr r_{13}+r_{13}\apr r_{23}-r_{23}\apr
r_{12},a^*\otimes b^*\otimes c^*\rangle -(\kappa+1)\langle \beta_{13}\apr
\beta_{23},a^*\otimes b^*\otimes c^*\rangle\\
&=&\langle r_{12}\apr r_{13}+r_{13}\apr r_{23}-r_{23}\apr
r_{12}-(\kappa+1)\frac{r_{13}+r_{31}}{2}\apr \frac{r_{23}+r_{32}}{2},a^*\otimes b^*\otimes
c^*\rangle.
\end{eqnarray*}
} So $r$ is a solution of the \MAYBE of \ewt $(\kappa+1)/4$ if and
only if $\alpha$ is an \tto $\calo$-operator with \bop $\beta$ of
\bwt $\kappa$.
\end{proof}

In the case when $\kappa=-1$, we have
\begin{coro} Let $\bfk$ be a field of characteristic not equal to 2. Let $A$ be a $\bfk$-algebra with finite $\bfk$-dimension and let $r\in A\otimes A$.
Define $\alpha$ and $\beta$ by Eq.~$($\mref{eq:alphabeta}$)$.
\begin{enumerate}
\item
If $\beta$ is invariant, then the following conditions are equivalent.
\begin{enumerate}
\item $r$ is a solution of the AYBE in
$A$.\mlabel{it:aybe}
\item
 $r$ $($resp. $-r^t$$)$ satisfies Eq.~$($\mref{eq:opweight1}$)$ $($resp. Eq.~$($\mref{eq:topweight1}$))$,
 i.e., $r$ $($resp. $-r^t$$)$ is an
$\calo$-operator of weight 1 associated to the $A$-bimodule
$\bfk$-algebra $(A^*,\rpr_+,R^*,L^*)$ $($resp.
$(A^*,\rpr_-,R^*,L^*)$$)$, where $A^*$ is equipped with the
associative algebra structure $\rpr_{+}$ $($resp. $\rpr_{-}$$)$
defined by Eq.~$($\mref{eq:betaproduct}$)$.\mlabel{it:oo} \item
$\alpha$ is an \tto $\calo$-operator with \bop $\beta$ of \bwt
$-1$. \mlabel{it:mo} \item For any $a^*,b^*\in A^*$,
\begin{equation}
(\alpha\pm\beta)(a^*\ast
b^*)=(\alpha\pm\beta)(a^*)\apr(\alpha\pm\beta)(b^*),\mlabel{eq:rrtho}
\end{equation}
where
\begin{equation}
a^* \ast b^*=R^*(r(a^*))b^*-a^*L^*(r^t(b^*)),\;\;\forall a^*,b^*\in
A^*.
\notag 
\end{equation}
\mlabel{it:rrtho}
\end{enumerate}
\mlabel{it:aybeq}
\item
When $r$ is skew-symmetric, then $r$ is a solution of
the AYBE in $A$ if and only if $r:A^*\to A$
is an $\mathcal O$-operator of weight zero.
\mlabel{it:aybss}
\end{enumerate}
\mlabel{co:aybea}
\end{coro}

\begin{proof}
If the symmetric part $\beta$ of $r$ is invariant, then by
Lemma~\mref{le:syin}, for any $a^*,b^*\in A^*$, we have
\begin{eqnarray*}
& &r(a^*)\apr r(b^*)-r(R^*(r(a^*))b^*-a^*L^*(r^t(b^*)))\\
&=&r(a^*)\apr
r(b^*)-r(R^*(r(a^*))b^*+a^*L^*(r(b^*))-2a^*L^*(\beta(b^*)))\\
&=&r(a^*)\apr r(b^*)-r(R^*(r(a^*))b^*+a^*L^*(r(b^*))+a^*\rpr_{+}
b^*),
\end{eqnarray*}
where the product $\rpr_+$ is defined by
Eq.~(\mref{eq:betaproduct}). So by Corollary~\mref{co:abas}, $r$
is a solution of the AYBE if and only if
Item (\mref{it:oo}) or (\mref{it:mo}) holds. Moreover, since for
any $a^*,b^*\in A^*$,
\begin{eqnarray*}
R^*(r(a^*))b^*-a^*L^*(r^t(b^*))&=&R^*(r(a^*))b^*+a^*L^*(r(b^*))+a^*\rpr_{+}
b^*\\
&=&R^*((-r^t)(a^*))b^*+a^*L^*((-r^t)(b^*))+a^*\rpr_{-} b^*,
\end{eqnarray*}
Eq.~(\mref{eq:rrtho}) is just a reformulation of
Eq.~(\mref{eq:opweight1}) and Eq.~(\mref{eq:topweight1}). So $r$ is a
solution of the AYBE if and only if
Item~(\mref{it:rrtho}) holds.
\smallskip

\noindent
(\mref{it:aybss}) This is the special case of Item~(\mref{it:aybeq}) when $\beta=0$.
\end{proof}

\subsection{$\calo$-operators and AYBE on Frobenius algebras}
\mlabel{ss:frob}
In this section we consider the relationship between $\calo$-operators and solutions of the AYBE on Frobenius algebras.

\begin{defn}
\begin{enumerate}
\item
Let $A$ be a $\bfk$-algebra and let $B(\;,\;): A\ot A\to \bfk$ be a nondegenerate bilinear form. Let $\varphi:A\rightarrow A^*$ denote
the induced injective linear map defined by
\begin{equation}
B(x,y)=\langle \varphi(x),y\rangle,\quad \forall x,y\in A. \mlabel{eq:definelinearmap}
\end{equation}
\item
A {\bf Frobenius $\bfk$-algebra} is a $\bfk$-algebra $(A,\apr\,)$ together with a nondegenerate bilinear form  $B(\;,\;): A\ot A \to \bfk$ that is invariant in the sense that
\begin{equation}
B(x\apr y,z)=B(x,y\apr z),\;\forall x,y,z\in A.
\notag 
\end{equation}
We use $(A,\apr,B)$
to denote a Frobenius $\bfk$-algebra. \item A Frobenius
$\bfk$-algebra is called {\bf symmetric} if
\begin{equation}
B(x,y)=B(y,x), \; \forall x,y\in A.
\notag 
\end{equation}
\item
A linear map $\beta: A\to A$ is called {\bf self-adjoint} (resp. {\bf skew-adjoint}) with respect to a bilinear form $B$ if for any $x,y\in A$, we have $B(\beta(x),y)=B(x,\beta(y))$ (resp. $B(\beta(x),y)=-B(x,\beta(y))$).
\end{enumerate}
\mlabel{de:frob}
\end{defn}
The study of Frobenius algebras was originated from the work~\mcite{Fro} of Frobenius and has found applications in broad areas of mathematics and physics. See~\mcite{Y, Bai2} for further details.

It is easy to get the
following result.
\begin{prop} \text{(\mcite{Y})} Let $A$ be a symmetric Frobenius $\bfk$-algebra with finite $\bfk$-dimension. Then the $A$-bimodule $(A,L,R)$ is isomorphic to the $A$-bimodule $(A^*,R^*,L^*)$. \mlabel{pp:frob}
\end{prop}

The following statement gives a class of symmetric Frobenius algebras from symmetric, invariant tensors.

\begin{coro}
Let $(A,\apr)$ be a $\bfk$-algebra with finite $\bfk$-dimension. Let $s\in A\otimes A$
be symmetric and invariant.
Suppose that $s$ regarded as a linear map from $A^*\to A$ is
invertible. Then $s^{-1}:A\to A^*$ regraded as a bilinear form
$B(\;,\;):A\otimes A\to \textbf{k}$ on $A$ through
Eq.~$($\mref{eq:definelinearmap}$)$ for $\varphi=s^{-1}$ is symmetric,
nondegenerate and invariant. Thus $(A,\apr, B)$ is a symmetric
Frobenius algebra.\mlabel{co:synoninbi}
\end{coro}

\begin{proof}
Since $s$ is symmetric and $s$ regarded as a linear map from $A^*$
to $A$ is invertible, $B(\;,\;)$ is symmetric and nondegenerate. On
the other hand, since $s$ is invariant, by Lemma~\mref{le:syin}
Eq.~(\mref{eq:dualssy}) holds. Thus, for any $x,y,z\in A$ and
$a^*=s^{-1}(x),b^*=s^{-1}(y),c^*=s^{-1}(z)$, we have
\begin{eqnarray*}
B(x\apr y,z)&=&\langle c^*,s(a^*)\apr s(b^*)\rangle=\langle
c^*L^*(s(a^*)),b^*\rangle\\
&=&\langle R^*(s(c^*))a^*,b^*\rangle=\langle a^*,s(b^*)\apr
s(c^*)\rangle=B(x,y\apr z),
\end{eqnarray*}
i.e., $B(\;,\;)$ is invariant. So the conclusion follows.
\end{proof}

\begin{lemma}
Let $(A,\apr,B)$ be a symmetric Frobenius $\bfk$-algebra with finite $\bfk$-dimension. Suppose
that $\beta:A\rightarrow A$ is an endomorphism of $A$ which is
self-adjoint with respect to $B$. Then
$\tilde{\beta}=\beta\varphi^{-1}:A^*\rightarrow A$ regarded as an
element of $A\otimes A$ is symmetric, where $\varphi:A\rightarrow
A^*$ is defined by Eq.~$($\mref{eq:definelinearmap}$)$. Moreover,
$\beta$ is a balanced $A$-bimodule homomorphism if and only if
$\tilde{\beta}$ is a balanced $A$-bimodule homomorphism.
\mlabel{le:frosy}
\end{lemma}

\begin{proof}
For any $a^*,b^*\in A^*$, set $x=\phi^{-1}(a^*), y=\phi^{-1}(b^*)$. Since $\beta$ is self-adjoint with respect to $B$, we see that
$$\langle \tilde{\beta}(a^*),b^*\rangle = \langle
\beta(x),\varphi(y)\rangle=B(\beta(x), y)=B(x,\beta(y))=\langle\varphi(x),\beta(y)\rangle
=\langle a^*,\tilde{\beta}(b^*)\rangle.$$
Therefore,
$\tilde{\beta}$ regarded as an element of $A\otimes A$ is symmetric.
Moreover, since $B$ is symmetric and invariant and $\beta$ is
self-adjoint with respect to $B$, for any $a^*,b^*\in A^*,z\in A$
and $x=\varphi^{-1}(a^*),y=\varphi^{-1}(b^*)$, we have
\begin{eqnarray*}
\langle R^*(\tilde{\beta}(a^*))b^*,z\rangle&=&\langle
R^*(\beta(x))\varphi(y),z\rangle=B(y,z\apr\beta(x))\\
\langle
a^*L^*(\tilde{\beta}(b^*)),z\rangle&=&\langle\varphi(x)L^*(\beta(y)),z\rangle=B(x,\beta(y)\apr
z)=B(\beta(y),z\apr x)=B(y,\beta(z\apr x)).
\end{eqnarray*}
Thus $\tilde{\beta}$ satisfies Eq.~(\mref{eq:dualssy}) if and only
if $\beta(z\apr x)=z\apr\beta(x)$, for any $x,z\in A$. On the
other hand,
\begin{eqnarray*}
\langle R^*(\tilde{\beta}(a^*))b^*,z\rangle&=&\langle
R^*(\beta(x))\varphi(y),z\rangle=B(y,z\apr\beta(x))=B(\beta(x),y\apr z)=B(x,\beta(y\apr z)),\\
\langle
a^*L^*(\tilde{\beta}(b^*)),z\rangle&=&\langle\varphi(x)L^*(\beta(y)),z\rangle=B(x,\beta(y)\apr
z).
\end{eqnarray*}
Therefore, $\tilde{\beta}$ satisfies Eq.~(\mref{eq:dualssy}) if and
only if $\beta(y\apr z)=\beta(y)\apr z$, for any $y,z\in A$. Hence
$\beta$ is an $A$-bimodule homomorphism if and only if
$\tilde{\beta}$ is an $A$-bimodule homomorphism. Furthermore,
the equivalence between Eq.~(\mref{eq:dualssy}) and
Eq.~(\mref{eq:dualsabi}) follows from Lemma~\mref{le:syin}.
\end{proof}

If $\beta={\rm id}$, then the above lemma says that
$\varphi^{-1}:A^*\rightarrow A$ is a balanced $A$-bimodule
homomorphism.

\begin{coro}
Let $(A,\apr,B)$ be a symmetric Frobenius $\bfk$-algebra of finite $\bfk$-dimension and let $\varphi:A\to A^*$ be the linear map defined by
Eq.~$($\mref{eq:definelinearmap}$)$. Suppose $\beta\in A\otimes A$
is symmetric. Then $\beta$ regarded as a linear map from
$(A^*,R^*,L^*)$ to $A$ is a balanced $A$-bimodule homomorphism if
and only if $\hat{\beta}=\beta\varphi:A\to A$ is a balanced
$A$-bimodule homomorphism.\mlabel{co:frosy1}
\end{coro}

\begin{proof}
Since $\beta\in A\otimes A$ is symmetric, for any $x,y\in A$,
\begin{eqnarray*}
&& \langle \beta,\varphi(x)\otimes\varphi(y)\rangle =\langle\beta,\varphi(y)\otimes\varphi(x)\rangle \\
&\Leftrightarrow&
\langle\beta(\varphi(x)),\varphi(y)\rangle =\langle\beta(\varphi(y)),\varphi(x)\rangle\\
&\Leftrightarrow&
B(\hat{\beta}(x),y)=B(\hat{\beta}(y),x).
\end{eqnarray*}
Thus $\hat{\beta}=\beta\varphi$ is self-adjoint with respect to
$B(\;,\;)$. So the conclusion follows from Lemma~\mref{le:frosy}.
\end{proof}

\begin{theorem}
Let $\bfk$ be a field of characteristic not equal to 2. Let $(A,\apr,B)$ be a symmetric Frobenius algebra of finite $\bfk$-dimension. Suppose that
$\alpha$ and $\beta$ are two endomorphisms of $A$ and $\beta$ is
self-adjoint with respect to $B$.
\begin{enumerate}
\item $\alpha$ is an \tto $\calo$-operator with \bop $\beta$ of
\bwt $\kappa$ if and only if
$\tilde{\alpha}:=\alpha\circ\varphi^{-1}:A^*\rightarrow A$ is an
\tto $\calo$-operator with \bop
$\tilde{\beta}:=\beta\circ\varphi^{-1}:A^*\rightarrow A$ of \bwt
$\kappa$, where the linear map $\varphi:A\rightarrow A^*$ is
defined by Eq.~$($\mref{eq:definelinearmap}$)$. \mlabel{it:tildealpha}
\item Suppose that in addition, $\alpha$ is skew-adjoint with
respect to $B$. Then $\tilde\alpha$ regarded as an element of
$A\otimes A$ is skew-symmetric and we have \mlabel{it:skewadjoint}
\begin{enumerate}
\item $r_{\pm}=\tilde{\alpha}\pm\tilde{\beta}$
regarded as an element of $A\otimes A$ is a solution of the \MAYBE of
\bwt $\frac{\kappa+1}{4}$ if and only if $\alpha$ is an \tto
$\calo$-operator with \bop $\beta$ of \type $k$. \mlabel{it:fk2}
\item
 If $\kappa =-1$, then $r_{\pm}=\tilde{\alpha}\pm\tilde{\beta}$ regarded as
an element of $A\otimes A$ is a solution of the AYBE if and only if
$\alpha$ is an \tto $\calo$-operator with \bop $\beta$ of \type
$-1$. \mlabel{it:fk1}
\item If $\kappa =0$, then $\tilde{\alpha}$ regarded as an element
of $A\otimes A$ is a solution of the AYBE if and only if $\alpha$ is
a Rota-Baxter operator of weight zero.\mlabel{it:fk3}
\end{enumerate}
\end{enumerate} \mlabel{thm:equivalence}
\end{theorem}

\begin{proof}
(\mref{it:tildealpha})
Since $B$ is symmetric and invariant, for any $x,y,z\in A$, we have
\begin{equation}
B(x\apr y,z)=B(x,y\apr z)\Leftrightarrow\langle\varphi(x\apr y),
z\rangle=\langle\varphi(x),y\apr z\rangle\Leftrightarrow\varphi(x
R(y))=\varphi(x)L^*(y),\mlabel{eq:invariant1}
\end{equation}
\begin{equation}
B(z,x\apr y)=B(y\apr z,x)\Leftrightarrow\langle\varphi(z),x\apr
y\rangle=\langle\varphi(y\apr z),x\rangle\Leftrightarrow
R^*(y)\varphi(z)=\varphi(L(y)z).\mlabel{eq:invariant2}
\end{equation}
On the other hand, since $\varphi$ is invertible, for any
$a^*,b^*\in A^*$, there exist $x,y\in A$ such that
$\varphi(x)=a^*,\varphi(y)=b^*$. So according to
Eq.~(\mref{eq:invariant1}) and Eq.~(\mref{eq:invariant2}), the
equation
$$\tilde{\alpha}(a^*)\apr\tilde{\alpha}(b^*)-\tilde{\alpha}(\varphi(\tilde{\alpha}(a^*)\apr
\varphi^{-1}(b^*)+\varphi^{-1}(a^*)\apr\tilde{\alpha}(b^*)))=\kappa \tilde{\beta}(a^*)\apr\tilde{\beta}(b^*),$$
 is equivalent to
$$\tilde{\alpha}(a^*)\apr\tilde{\alpha}(b^*)-\tilde{\alpha}(R^*(\tilde{\alpha}(a^*))b^*
+a^*L^*(\tilde{\alpha}(b^*)))=\kappa \tilde{\beta}(a^*)\apr\tilde{\beta}(b^*).$$

By Lemma~\mref{le:frosy}, $\beta:A\rightarrow A$ is a balanced
$A$-bimodule homomorphism if and only if
$\tilde{\beta}:A^*\rightarrow A$ is a balanced $A$-bimodule
homomorphism. So $\alpha$ is an \tto $\calo$-operator with \bop
$\beta$ of \bwt $\kappa$ if and only if $\tilde{\alpha}$ is an
\tto $\calo$-operator with \bop $\tilde{\beta}$ of \bwt $\kappa$.
\smallskip

\noindent
(\mref{it:skewadjoint})
If $\alpha$ is skew-adjoint with
respect to $B$, then $$\langle
\alpha(x),\varphi(y)\rangle+\langle\varphi(x),\alpha(y)\rangle=0,\;\;\forall
x,y\in A.$$ Hence $\langle\tilde{\alpha}(a^*),b^*\rangle+\langle
a^*,\tilde{\alpha}(b^*)\rangle=0$ for any $a^*,b^*\in A^*$. So
$\tilde{\alpha}$ regarded as an element of $A\otimes A$ is
skew-symmetric.
\smallskip

Then by Theorem~\mref{thm:aybea}, Item~(\mref{it:fk2}) holds. By
Corollary~\mref{co:aybea} Item~(\mref{it:fk1}) and
Item~(\mref{it:fk3}) hold.
\end{proof}

\begin{coro}
Let $\bfk$ be a field of characteristic not equal to 2. Let $A$ be a $\bfk$-algebra of finite $\bfk$-dimension and let $r\in A\otimes A$. Define
$\alpha, \beta\in A\otimes A$ by Eq.~$($\mref{eq:alphabeta}$)$. Then
$r=\alpha+\beta$. Let $B:A\otimes A\to\textbf{k}$ be a nondegenerate
symmetric and invariant bilinear form. Define the linear map
$\varphi:A\to A^*$ by Eq.~$($\mref{eq:definelinearmap}$)$.
\begin{enumerate}
\item Suppose that $\beta\in A\otimes A$ is invariant. Then
$r$ is a solution of the \MAYBE of \ewt $\frac{\kappa+1}{4}$ if and
only if $\hat{\alpha}=\alpha\varphi: A\to A$ is an \tto
$\calo$-operator with \bop $\hat{\beta}=\beta\varphi:A\to A$ of
\type $k$.
\item Suppose that $\beta\in A\otimes A$ is invariant. Then $r$ is
a solution of the AYBE if and only if $\hat{\alpha}=\alpha\varphi:
A\to A$ is an \tto $\calo$-operator with \bop
$\hat{\beta}=\beta\varphi:A\to A$ of \type $-1$. In particularly, if
in addition, $\beta=0$, i.e., $r$ is skew-symmetric, then $r$ is a
solution of the AYBE if and only if $\hat{\alpha}=\hat
r=r\varphi:A\rightarrow A$  is a Rota-Baxter operator of weight
zero.
\end{enumerate}
\mlabel{co:equivalence1}
\end{coro}

\begin{proof}
Since $\alpha\in A\otimes A$ is skew-symmetric, for any $x,y\in A$,
\begin{eqnarray*}
&&\langle \alpha,\varphi(x)\otimes\varphi(y)\rangle =-\langle\alpha,\varphi(y)\otimes\varphi(x)\rangle\\
&\Leftrightarrow&
\langle\alpha(\varphi(x)),\varphi(y)\rangle= -\langle\alpha(\varphi(y)),\varphi(x)\rangle \\
&\Leftrightarrow&
B(\hat{\alpha}(x),y)=-B(\hat{\alpha}(y),x).
\end{eqnarray*}
Thus
$\hat{\alpha}=\alpha\varphi$ is skew-adjoint with respect to
$B(\;,\;)$. On the other hand, from the proof of
Corollary~\mref{co:frosy1}, we show that $\hat{\beta}=\beta\varphi$
is self-adjoint with respect to $B(\;,\;)$. So the conclusion
follows from Theorem~\mref{thm:equivalence}.
\end{proof}

\subsection{\Tto $\calo$-operators in general and \MAYBE}
\mlabel{ss:goaybe}

We now
establish the relationship between an \tto $\calo$-operator $\alpha: V\to A$ in general and the AYBE. For this purpose we prove
that an \tto $\calo$-operator $\alpha:V\to A$ naturally gives rise
to an \tto $\calo$-operator on a larger associative algebra $\hat
A$ associated to the dual bimodule $({\hat A}^*, R^*_{\hat A},
L^*_{\hat A})$. We first introduce some notations.

\begin{defn}
Let $A$ be a $\bfk$-algebra and let $(V,\ell,r)$ be an $A$-bimodule, both with finite $\bfk$-dimension.
Let $(V^*,r^*,\ell^*)$ be the dual $A$-bimodule and let $\hat A=
A\ltimes_{r^*,\ell^*}V^*$. Identify a linear map
$\gamma:V\rightarrow A$ as an element in $\hat A\otimes \hat A$
through the injective map
\begin{equation}
\Hom_\bfk(V,A) \cong A \ot V^* \hookrightarrow \hat{A} \ot
\hat{A}.
\mlabel{eq:id2}
\end{equation}
Denote
\begin{equation}
\tilde{\gamma}_\pm:=\gamma\pm \gamma^{21},
\mlabel{eq:syco}
\end{equation}
where $\gamma^{21}=\sigma(\gamma) \in V^* \ot A \subset \hat{A}
\ot \hat{A}$ with $\sigma:A\ot V^* \to V^* \ot A, a\ot u^* \to u^*\ot a,$ being the switch operator. \mlabel{de:tilde}
\end{defn}

\begin{lemma}
Let $A$ be a $\bfk$-algebra and let $(V,\ell,r)$ be an $A$-bimodule, both with finite $\bfk$-dimension.
Suppose that $\beta:V\rightarrow A$ is a linear map which is
identified as an element in $\hat A\otimes \hat A$ by Eq.~$($\mref{eq:id2}$)$. Define
$\tilde{\beta}_+$ by Eq.~$($\mref{eq:syco}$)$. Then $\tilde{\beta}_+$, identified as a linear map from ${\hat A}^*$ to $\hat A$, is a balanced
$\hat A$-bimodule homomorphism from $({\hat
A}^*,R^*_{\hat A}, L^*_{\hat A})$ to $(\hat{A}, L_{\hat{A}},R_{\hat{A}})$ if and only if $\beta:V\to A$ is a balanced $A$-bimodule
homomorphism from $(V,\ell,r)$ to $(A,L_A,R_A)$. \mlabel{le:syco}
\end{lemma}
\begin{proof}
For the linear map $\tilde{\beta}_+:\hat{A}^*\to \hat{A}$, we have
$\tilde{\beta}_+(a^*)=\beta^*(a^*)$ for $a^*\in A^*$ and
$\tilde{\beta}_+(u)=\beta(u)$ for $u\in V$, where $\beta^*:A^*\rightarrow V^*$ is
the dual linear map associated to $\beta$ given by
$$\langle \beta^*(a^*), v\rangle=\langle a^*,\beta(v)\rangle,\;\;\forall
a^*\in A^*, v\in V.$$ First suppose that $\beta:(V,\ell,r)\to A$ is a balanced $A$-bimodule homomorphism. Let
$b^*\in A^*, v\in V$, then
$$R_{\hat A}^*(\tilde{\beta}_+(a^*+u))(b^*+v)=R_{\hat A}^*(\beta^*(a^*))b^*+R_{\hat A}^*(\beta^*(a^*))v+R_{\hat A}^*(\beta(u))b^*+R_{\hat A}^*(\beta(u))v,$$
$$(a^*+u)L_{\hat A}^*(\tilde{\beta}_+(b^*+v))=a^*L_{\hat
A}^*(\beta^*(b^*))+a^*L_{\hat A}^*(\beta(v))+uL_{\hat
A}^*(\beta^*(b^*))+uL_{\hat A}^*(\beta(v)).$$ On the other hand, for
any $x\in A$, $w^*\in V^*$,
{\allowdisplaybreaks
{\small \begin{eqnarray*} &&\langle R_{\hat
A}^*(\beta^*(a^*))b^*-a^*L_{\hat A}^*(\beta^*(b^*)),x\rangle
=\langle b^*,x\apr\beta^*(a^*)\rangle
-\langle a^*,\beta^*(b^*)\apr x\rangle =0,\\
&&\langle R_{\hat A}^*(\beta^*(a^*))b^*-a^*L_{\hat
A}^*(\beta^*(b^*)),w^*\rangle =\langle
b^*,w^*\apr\beta^*(a^*)\rangle -\langle a^*,\beta^*(b^*)\apr
w^*\rangle =0,\\
& &\langle R_{\hat A}^*(\beta^*(a^*))v-a^*L_{\hat
A}^*(\beta(v)),x\rangle =\langle v,x\apr\beta^*(a^*)\rangle
-\langle a^*,\beta(v)\apr x\rangle =\langle
a^*,\beta(vr(x))-\beta(v)\apr x\rangle =0,\\
&& \langle R_{\hat A}^*(\beta^*(a^*))v-a^*L_{\hat
A}^*(\beta(v)),w^*\rangle =\langle
v,w^*\apr\beta^*(a^*)\rangle-\langle a^*,\beta(v)\apr w^*\rangle
=0,\\
& &\langle R_{\hat A}^*(\beta(u))b^*-uL_{\hat
A}^*(\beta^*(b^*)),x\rangle =\langle b^*,x\apr\beta(u)\rangle
-\langle u,\beta^*(b^*)\apr x\rangle =\langle
b^*,x\apr\beta(u)-\beta(l(x)u)\rangle =0,\\
 && \langle R_{\hat
A}^*(\beta(u))b^*-uL_{\hat A}^*(\beta^*(b^*)),w^*\rangle =\langle
b^*,w^*\apr\beta(u)\rangle-\langle
u,\beta^*(b^*)\apr w^*\rangle =0.\\
&&\langle R_{\hat A}^*(\beta(u))v-uL_{\hat A}^*(\beta(v)),w^*\rangle
= \langle v,w^*\apr\beta(u)\rangle -\langle u,\beta(v)\apr
w^*\rangle =\langle \ell(\beta(u))v-ur(\beta(v)),w^*\rangle =0,\\
&&\langle R_{\hat A}^*(\beta(u))v-uL_{\hat A}^*(\beta(v)),x\rangle
=\langle v,x\apr\beta(u)\rangle -\langle u,\beta(v)\apr x\rangle
=0.
\end{eqnarray*}}
}
 Therefore, $R_{\hat
A}^*(\tilde{\beta}_+(a^*+u))(b^*+v)=(a^*+u)L_{\hat
A}^*(\tilde{\beta}_+(b^*+v))$. Since $\tilde{\beta}_{+}\in \hat{A}\otimes\hat{A}$ is symmetric, by
Lemma~\mref{le:syin}, $\tilde{\beta}_+$, identified as a linear map from ${\hat A}^*$ to $\hat A$, is a balanced
$\hat A$-bimodule homomorphism from $({\hat
A}^*,R^*_{\hat A}, L^*_{\hat A})$ to $(\hat{A}, L_{\hat{A}},R_{\hat{A}})$.

Conversely, if $\tilde{\beta}_+$, identified as a linear map from ${\hat A}^*$ to $\hat A$, is a balanced
$\hat A$-bimodule homomorphism from $({\hat
A}^*,R^*_{\hat A}, L^*_{\hat A})$ to $(\hat{A}, L_{\hat{A}},R_{\hat{A}})$, then for any $u,v\in V,x\in A$,
\begin{eqnarray*}
R_{\hat{A}}^*(\tilde{\beta}_{+}(u))v =uL_{\hat{A}}^*(\tilde{\beta}_{+}(v))&\Leftrightarrow&
\ell(\beta(u))v=ur(\beta(v)),\\
\tilde{\beta}_{+}(R_{\hat{A}}^*(x)v) =x\apr\tilde{\beta}_{+}(v) &\Leftrightarrow & \beta(\ell(x)v)=x\apr\beta(v),\\
\tilde{\beta}_{+}(uL_{\hat{A}}^*(x))=\tilde{\beta}_{+}(u)\apr
x&\Leftrightarrow&\beta(ur(x))=\beta(u)\apr x.
\end{eqnarray*}
  So
$\beta:(V,\ell,r)\to (A,L_A,R_A)$ is a balanced $A$-bimodule homomorphism.
\end{proof}

\begin{theorem}
Let $A$ be a $\bfk$-algebra and let $(V,\ell,r)$ be an $A$-bimodule, both with finite $\bfk$-dimension. Let
$\alpha,\beta:V\to A$ be two $\bfk$-linear maps.
Let $\tilde{\alpha}_-$ and $\tilde{\beta}_+$ be defined by Eq.~$($\mref{eq:id2}$)$ and identified as linear maps from $\hat{A}^*$ to $\hat{A}$.
Then $\alpha$ is an \tto
$\calo$-operator with \bop $\beta$ of \bwt $\kappa$ if and only if
$\tilde{\alpha}_-$ is an \tto $\calo$-operator with \bop $\tilde{\beta}_+$ of \bwt $\kappa$. \mlabel{thm:skewgm}
\end{theorem}
\begin{proof}
Note that for any $a^*\in A^*,v\in V$,
$\tilde{\alpha}_-(a^*)=\alpha^*(a^*)$ and
$\tilde{\alpha}_-(v)=-\alpha(v)$, where $\alpha^*:A^*\rightarrow
V^*$ is the dual linear map of $\alpha$. Suppose that $\alpha$ is
an \tto $\calo$-operator with \bop $\beta$ of \bwt $\kappa$. Then
for any $a^*,b^*\in A^*$, $u,v\in V$, we have
\begin{eqnarray*}
&
&\tilde{\alpha}_-(u+a^*)\apr\tilde{\alpha}_-(v+b^*)-\tilde{\alpha}_-(R_{\hat
A}^*(\tilde{\alpha}_-(u+a^*))(v+b^*)+
(u+a^*)L_{\hat A}^*(\tilde{\alpha}_-(v+b^*)))\\
&=&\alpha(u)\apr\alpha(v)-\alpha(u)\apr\alpha^*(b^*)-\alpha^*(a^*)\apr\alpha(v)+
\alpha^*(a^*)\apr\alpha^*(b^*)\\
& &-\tilde{\alpha}_-(-R_{\hat A}^*(\alpha(u))v-R_{\hat
A}^*(\alpha(u))b^*+R_{\hat A}^*(\alpha^*(a^*))v+
R_{\hat A}^*(\alpha^*(a^*))b^*-uL_{\hat A}^*(\alpha(v))\\ & &-a^*L_{\hat A}^*(\alpha(v))+uL_{\hat A}^*(\alpha^*(b^*))+a^*L_{\hat A}^*(\alpha^*(b^*)))\\
&=&\alpha(u)\apr\alpha(v)-\alpha(\ell(\alpha(u))v)-\alpha(ur(\alpha(v)))
-r^*(\alpha(u))\alpha^*(b^*)+\alpha^*(R_{\hat A}^*(\alpha(u))b^*)\\
& &-\alpha^*(uL_{\hat A}^*(\alpha^*(b^*)))-
\alpha^*(a^*)\ell^*(\alpha(v))-\alpha^*(R_{\hat
A}^*(\alpha^*(a^*))v)+\alpha^*(a^*L_{\hat A}^*(\alpha(v))).
\end{eqnarray*}
On the other hand, for any $w\in V$ we have
\begin{eqnarray*}
& &\langle -r^*(\alpha(u))\alpha^*(b^*)+\alpha^*(R_{\hat
A}^*(\alpha(u))b^*)-
\alpha^*(uL_{\hat A}^*(\alpha^*(b^*))),w\rangle \\
&=&\langle
b^*,\alpha(w)\apr\alpha(u)-\alpha(\ell(\alpha(w))u+wr(\alpha(u))\rangle =\langle b^*,\kappa \beta(w)\apr\beta(u)\rangle \\
&=&\langle b^*,\kappa \beta(wr(\beta(u)))\rangle =\langle
\kappa r^*(\beta(u))\beta^*(b^*),w\rangle.
\end{eqnarray*}
Therefore $$-r^*(\alpha(u))\alpha^*(b^*)+\alpha^*(R_{\hat
A}^*(\alpha(u))b^*)- \alpha^*(uL_{\hat
A}^*(\alpha^*(b^*)))=\kappa r^*(\beta(u))\beta^*(b^*).$$ Similarly,
{\allowdisplaybreaks
\begin{eqnarray*}
& &\langle -\alpha^*(a^*)\ell^*(\alpha(v))-\alpha^*(R_{\hat A}^*(\alpha^*(a^*))v)
+\alpha^*(a^*L_{\hat A}^*(\alpha(v))),w\rangle \\
&=&\langle
a^*,\alpha(v)\apr\alpha(w)-\alpha(\ell(\alpha(v))w+vr(\alpha(w)))\rangle =\langle a^*,\kappa \beta(v)\apr\beta(w)\rangle \\
&=&\langle a^*,\kappa \beta(\ell(\beta(v))w)\rangle =\langle
\kappa \beta^*(a^*)\ell^*(\beta(v)),w\rangle.
\end{eqnarray*}
}
Hence $$-\alpha^*(a^*)\ell^*(\alpha(v))-\alpha^*(R_{\hat
A}^*(\alpha^*(a^*))v)+\alpha^*(a^*L_{\hat A}^*(\alpha(v)))=
\kappa \beta^*(a^*)\ell^*(\beta(v)).$$ So
\begin{eqnarray*}
&
&\tilde{\alpha}_-(u+a^*)\apr\tilde{\alpha}_-(v+b^*)-\tilde{\alpha}_-(R_{\hat
A}^*(\tilde{\alpha}_-(u+a^*))(v+b^*)+
(u+a^*)L_{\hat A}^*(\tilde{\alpha}_-(v+b^*)))\\
&=&\kappa \beta(u)\apr\beta(v)+\kappa r^*(\beta(u))\beta^*(b^*)+\kappa \beta^*(a^*)\ell^*(\beta(v))\\
&=&\kappa \beta(u)\apr\beta(v)+\kappa \beta(u)\apr\beta^*(b^*)+\kappa \beta^*(a^*)\apr\beta(v)
=\kappa \tilde{\beta}_+(u+a^*)\tilde{\beta}_+(v+b^*).
\end{eqnarray*}
If $\kappa =0$, then the above equation implies that
$\tilde{\alpha}_-$ is an $\calo$-operator of weight zero. If
$\kappa \neq 0$, then $\beta$ is a balanced $A$-bimodule
homomorphism, which, according to Lemma~\mref{le:syco}, implies
that $\tilde{\beta}_{+}$ from $({\hat A}^*,R^*_{\hat A}, L^*_{\hat
A})$ to $\hat A$ is a balanced $\hat A$-bimodule homomorphism. So
$\tilde{\alpha}_-$ is an \tto $\calo$-operator with \bop
$\tilde{\beta}_+$ of \bwt $\kappa$.

Conversely, if $\tilde{\alpha}_-$ is an \tto $\calo$-operator with
\bop $\tilde{\beta}_+$ of \bwt $\kappa$. If $\kappa \neq 0$, then
$\tilde{\beta}_{+}$ from $({\hat A}^*,R^*_{\hat A}, L^*_{\hat A})$
to $\hat A$ is a balanced $\hat A$-bimodule homomorphism, which by
Lemma~\mref{le:syco} implies that $\beta$ from $(V,\ell,r)$ to $A$
is a balanced $A$-bimodule homomorphism. Moreover, for any $u,v\in
V$ we have
\begin{equation}
\tilde{\alpha}_-(u)\apr\tilde{\alpha}_-(v)-\tilde{\alpha}_-(R_{\hat
A}^*(\tilde{\alpha}_-(u))v+ uL_{\hat
A}^*(\tilde{\alpha}_-(v)))=\kappa \tilde{\beta}_+(u)\tilde{\beta}_+(v),\mlabel{eq:biggeralpha}
\end{equation}
which implies that $\alpha$ is an \tto $\calo$-operator with \bop
$\beta$ of \bwt $\kappa$. If $\kappa =0$, then
Eq.~(\mref{eq:biggeralpha}) for $\kappa =0$ implies that $\alpha$
is an $\calo$-operator of weight zero. So the conclusion follows.
\end{proof}

By Theorem~\mref{thm:skewgm}, the results from the previous sections on $\calo$-operators on $A$ can be applied to general $\calo$-operators.

\begin{coro}
Let $A$ be a $\bfk$-algebra and let $V$ be an $A$-bimodule, both with finite $\bfk$-dimension.
\begin{enumerate}
 \item Suppose the characteristic of the field $\bfk$ is not 2.
Let $\alpha, \beta:V\rightarrow A$ be linear maps which are
identified as elements in
$(A\ltimes_{r^*,\ell^*}V^*)\otimes(A\ltimes_{r^*,\ell^*}V^*)$. Then
$\alpha$ is an \tto $\calo$-operator with \bop $\beta$ of \bwt $k$
if and only if $(\alpha-\alpha^{21})\pm (\beta+\beta^{21})$ is a
solution of the \MAYBE of \bwt $\frac{\kappa+1}{4}$ in
$A\ltimes_{r^*,\ell^*}V^*$. \mlabel{it:motoaybe6}
\item Let
$\alpha:V\rightarrow A$ be a linear map which is identified as an
element in
$(A\ltimes_{r^*,\ell^*}V^*)\otimes(A\ltimes_{r^*,\ell^*}V^*)$.
Then $\alpha$ is an $\calo$-operator of weight zero if and only if
$\alpha-\alpha^{21}$ is a skew-symmetric solution of the AYBE in
Eq.~$($\mref{eq:aybe2}$)$ in $A\ltimes_{r^*,\ell^*}V^*$. In
particular, a linear map $P:A\rightarrow A$ is a Rota-Baxter
operator of weight zero if and only if $r=P-P^{21}$ is a
skew-symmetric solution of the AYBE in
$A\ltimes_{R^*,L^*}A^*$. \mlabel{it:motoaybe5} \item Let $\alpha, \beta:V\rightarrow
A$ be two linear maps which are identified as elements in
$(A\ltimes_{r^*,\ell^*}V^*)\otimes(A\ltimes_{r^*,\ell^*}V^*)$.
Then $\alpha$ is an \tto $\calo$-operator with \bop $\beta$ of
\type $-1$ if and only if $(\alpha-\alpha^{21})\pm
(\beta+\beta^{21})$ is a solution of the AYBE in
$A\ltimes_{r^*,\ell^*}V^*$. \mlabel{it:motoaybe1}

\item Let $\aop:A\rightarrow A$ be a linear map which is
identified as an element in $(A\ltimes_{R^*,L^*}A^*)\otimes
(A\ltimes_{R^*,L^*}A^*)$. Then $\aop$ satisfies
Eq.~$($\mref{eq:-1myb}$)$ if and only if $(\aop-\aop^{21})\pm({\rm
id}+{\rm id^{21}})$ is a solution of the AYBE in
$A\ltimes_{R^*,L^*}A^*$. \mlabel{it:motoaybe3}
\item Let
$P:A\rightarrow A$ be a linear map which is identified as an element
of $A\ltimes_{R^*,L^*} A^*$. Then $P$ is a Rota-Baxter operator of
weight $\lambda\neq 0$ if and only if both
$\frac{2}{\lambda}(P-P^{21})+2{\rm id}$ and
$\frac{2}{\lambda}(P-P^{21})-2{\rm id}^{21}$ are solutions of the
AYBE in $A\ltimes_{R^*,L^*}A^*$.
\mlabel{it:motoaybe4}
\end{enumerate}
\mlabel{co:motoaybe1}
\end{coro}

\begin{proof}
\noindent (\mref{it:motoaybe6}) This follows from
Theorem~\mref{thm:skewgm} and Theorem~\mref{thm:aybea}.
\medskip

\noindent
(\mref{it:motoaybe5}) This follows from Theorem~\mref{thm:skewgm}
for $\kappa =0$ (or $\beta=0$)  and Corollary~\mref{co:aybea}.
\medskip

\noindent (\mref{it:motoaybe1}) This follows from
Theorem~\mref{thm:skewgm} for $\kappa =-1$ and
Corollary~\mref{co:aybea}.
\medskip

\noindent (\mref{it:motoaybe3}) This follows from Item~(\mref{it:motoaybe1})
in the case that $(V,r,\ell)=(A,L,R)$ and $\beta={\rm id}$.

\medskip

\noindent (\mref{it:motoaybe4}) By~\cite{E1} (see the discussion
after Corollary~\mref{co:mop}), $P$ is a Rota-Baxter operator of
weight $\lambda\neq 0$ if and only if $\frac{2P}{\lambda}+{\rm
id}$ is an \tto $\calo$-operator with \bop ${\rm id}$ of \bwt $-1$
from $(A,L,R)$ to $A$, i.e., $\frac{2P}{\lambda}+{\rm id}$
satisfies Eq.~(\mref{eq:-1myb}). Then the conclusion follows from
Item~(\mref{it:motoaybe3}).

\end{proof}

\section{\Tto $\calo$-operators and the \gyb associative Yang-Baxter equation}
\mlabel{sec:maybe}
We define the generalized associative Yang-Baxter equation and study its relationship with \tto $\calo$-operators.
\subsection{\Gyb associative Yang-Baxter equation}
We adapt the same notations as in Definition~\mref{de:ayb2}.
\begin{defn}
An element $r\in A\ot A$ is called a solution of the {\bf \gyb
associative Yang-Baxter equation (\GAYBE) in $A$} if it satisfies the
relation
\begin{equation}
( \id \otimes  \id \otimes L(x)-R(x)\otimes  \id \otimes
 \id )(r_{12}r_{13}+r_{13}r_{23}-r_{23}r_{12})=0,\;\;\forall x\in
 A.
\mlabel{eq:maybe}
\end{equation}
\mlabel{de:maybe}
\end{defn}

%
Our definition is motivated by
the following fact (also see Proposition 5.1 in \cite{Ag1}) which
relates to the construction of a kind of bialgebras in the different
notions such that associative D-bialgebras \mcite{Z}, balanced
infinitesimal bialgebras (in the opposite algebras) \mcite{Ag3} and
antisymmetric infinitesimal bialgebras \mcite{Bai2}.

\begin{prop} {\rm (}\mcite{Ag1,Ag3,Bai2}{\rm )}
Let $A$ be a $\bfk$-algebra with finite $\bfk$-dimension and let $r\in A\otimes A$. Define
$\Delta:A\rightarrow A\otimes A$ by
\begin{equation}
\Delta (x)=( \id \otimes L(x)-R(x)\otimes  \id )r,\;\;\forall x\in
A.\mlabel{eq:coproduct}
\end{equation}
Then
\begin{equation}
\Delta^*:A^*\otimes A^*\hookrightarrow (A\ot A)^* \rightarrow A^*
\mlabel{eq:productast}
\end{equation}
defines an
associative multiplication on $A^*$ if and only if $r$ is a solution of the \GAYBE.\mlabel{pp:bialgebra}
\end{prop}

Moreover, we have the following conclusion
\begin{lemma}
Let $(A,\cdot)$ be a $\bfk$-algebra with finite $\bfk$-dimension. Let $r\in A\otimes A$. The
multiplication $\ast$ on $A^*$ defined by
Eq.~$($\mref{eq:productast}$)$ is also given by
\begin{equation}
a^* \ast b^*=R^*(r(a^*))b^*-L^*(r^t(b^*))a^*,\;\;\forall a^*,b^*\in
A^*.\mlabel{eq:productr}
\end{equation}
 \mlabel{lemma:maybe}
\end{lemma}
\begin{proof}
 Let $\{e_1,...,e_n\}$ be a basis of $A$ and
$\{e_1^*,...,e_n^*\}$ be its dual basis. Suppose that
$r=\sum_{i,j}a_{i,j}e_i\otimes e_j$ and $e_i\cdot
e_j=\sum_{k}c_{i,j}^ke_k$. Then for any $k,l$ we have
\begin{eqnarray*}
e_k^* \ast e_l^* &=& \sum_s\langle e_k^*\otimes
e_l^*,\Delta(e_s)\rangle e_s^* = \sum_s\langle e_k^*\otimes e_l^*,(\id\otimes L(e_s)-R(e_s)\otimes\id)r\rangle e_s^*\\
&=& \sum_{s,t}(a_{k,t}c_{s,t}^l-c_{t,s}^ka_{t,l})e_s^*=
R^*(r(e_k^*))e_l^*-L^*(r^t(e_l^*))e_k^*.
\end{eqnarray*}
\end{proof}

The above lemma motivates us to apply the approach considered in
Section~\mref{ss:moop}. More precisely, we take
the $A$-bimodule $\bfk$-algebra $(R,\rpr,\ell,r)$ to be $(A^*,R^*,L^*)$ with the zero multiplication
and set
\begin{equation}
\delta_{+}=r,\quad \delta_{-}=-r^t.
\end{equation}
Assume that $\bfk$ has characteristic not equal to 2 and define
\begin{equation}
\alpha=(r-r^t)/2,\quad \beta=(r+r^t)/2,
\mlabel{eq:alphabeta1}
\end{equation}
that is, $\alpha$ and $\beta$ are the {\bf skew-symmetric part} and
the {\bf symmetric part} of $r$ respectively. So $r=\alpha+\beta$
and $r^t=-\alpha+\beta$.

\begin{prop}
Let $\bfk$ have characteristic not equal to 2. Let $(A,\apr)$ be a $\bfk$-algebra with finite $\bfk$-dimension and $r\in A\otimes A$. Let
$\alpha,\beta$ be given by Eq.~$($\mref{eq:alphabeta1}$)$. Suppose that
$\beta$ is a balanced $A$-bimodule homomorphism, that is, $\beta$ satisfies
Eq.~$($\mref{eq:invariant}$)$. If $\alpha$ is an \tto $\calo$-operator
with \bop $\beta$ of any \bwt $\kappa\in \bfk$, then the product defined by
Eq.~$($\mref{eq:productr}$)$ defines a $\bfk$-algebra structure on
$A^*$ and
$r$ is a solution of the \GAYBE. \mlabel{pp:mtoop}
\end{prop}
\begin{proof}
Applying Theorem~\mref{thm:ansatz} to the $A$-bimodule $\bfk$-algebra $(R,\rpr,\ell,r)$ constructed before the proposition, we see that the product defined by Eq.~(\mref{eq:productr}) is associative. Then by Lemma~\mref{lemma:maybe}, $r$ is a solution of the \GAYBE.
\end{proof}

As a direct consequence, we have
\begin{coro}
Under the same assumption as in Proposition~\mref{pp:mtoop}, a solution of the \MAYBE of any \bwt $\kappa\in \bfk$ is also a solution of the \GAYBE.
\mlabel{co:II-MYBE}
\end{coro}

\begin{proof}
Let $r$ be a solution of the \MAYBE of \bwt $\kappa$. Define $\alpha$ and $\beta$ by Eq.~(\mref{eq:alphabeta1}). Then by Theorem~\mref{thm:aybea}, $\alpha$ is an \tto $\calo$-operator with \bop $\beta$ of \type $4\kappa-1$. Then by Proposition~\mref{pp:mtoop}, $r$ is a solution of the \GAYBE.
\end{proof}

\subsection{Relation with \tto $\calo$-operators}

\begin{lemma}
Let $A$ be a $\bfk$-algebra and $(V,\ell,r)$ be a bimodule. Let
$\alpha:V\to A$ be a linear map. Then the product
\begin{equation}
u\ast_\alpha v:=\ell(\alpha(u))v+ur(\alpha(v)),\quad \forall u,v\in V,
\end{equation}
defines a $\bfk$-algebra structure on $V$ if and only if the
following equation holds:
\begin{equation}
\ell\Big(\alpha(u)\apr\alpha(v)-\alpha(u\ast_\alpha v)\Big)w=
ur\Big(\alpha(v)\apr\alpha(w)-\alpha(v\ast_\alpha w)\Big),\quad
\forall u,v\in V. \mlabel{eq:biaspro}
\end{equation}
\mlabel{le:biaspro}
\end{lemma}

\begin{proof}
It follows from Lemma~\mref{le:product} by setting $(R,\ell,r)=(V,\ell,r)$
and $\lambda=0$.
\end{proof}

\begin{theorem}
Let $A$ be a $\bfk$-algebra and $(V,\ell,r)$ be an $A$-bimodule, both of finite dimension over $\bfk$. Let
$\alpha:V\rightarrow A$ be a linear map from $V$ to $A$. Using the
same notations in Definition~\mref{de:tilde}, then $\tilde{\alpha}_{-}$
identified as an element of $\hat{A}\otimes\hat{A}$ is a
skew-symmetric solution of the \GAYBE~$($\mref{eq:maybe}$)$ if and only if Eq.~$($\mref{eq:biaspro}$)$ and the
following equations hold:
\begin{equation}
\alpha(u)\apr\alpha(\ell(x)v)-\alpha(u\ast_\alpha (\ell(x)v))=
\alpha(ur(x))\apr\alpha(v)-\alpha((ur(x))\ast_\alpha v),\mlabel{eq:uvx1}
\end{equation}
\begin{equation}
\alpha(u)\apr\alpha(vr(x))-\alpha(u\ast_\alpha (vr(x)))=(\alpha(u)\apr\alpha(v))\apr
x-\alpha(u\ast_\alpha v)\apr x,\mlabel{eq:uvx2}
\end{equation}
\begin{equation}
\alpha(\ell(x)u)\apr\alpha(v)-\alpha((\ell(x)u)\ast_\alpha v)=
x\apr(\alpha(u)\apr\alpha(v))-x\apr\alpha(u\ast_\alpha v),\mlabel{eq:uvx3}
\end{equation}
for any $u,v\in V,x\in A$.\mlabel{thm:maybeequi}
\end{theorem}
\begin{proof}
By Proposition~\mref{pp:bialgebra}, Lemma~\mref{lemma:maybe} and
Lemma~\mref{le:biaspro} we see that $\tilde{\alpha}_{-}\in\hat{A}\otimes\hat{A}$ is a skew-symmetric
solution of the \GAYBE~(\mref{eq:maybe}) if and only if for any $u,v,w\in
V,a^*,b^*,c^*\in A^*$,
{\small
\begin{eqnarray*}
&
&R_{\hat{A}}^*(\tilde{\alpha}_{-}(u+a^*)\apr\tilde{\alpha}_{-}(v+b^*)-
\tilde{\alpha}_{-}(R_{\hat{A}}^*(\tilde{\alpha}_{-}(u+a^*))(v+b^*)+
(u+a^*)L_{\hat{A}}^*(\tilde{\alpha}_{-}(v+b^*))))(w+c^*)\\
&=&
(u+a^*)L_{\hat{A}}^*(\tilde{\alpha}_{-}(v+b^*)\apr\tilde{\alpha}_{-}(w+c^*)-
\tilde{\alpha}_{-}(R_{\hat{A}}^*(\tilde{\alpha}_{-}(v+b^*))(w+c^*)+
(v+b^*)L_{\hat{A}}^*(\tilde{\alpha}_{-}(w+c^*)))),
\end{eqnarray*}
}
By the proof of Theorem~\mref{thm:skewgm}, the above equation is equivalent to
{\allowdisplaybreaks
\begin{eqnarray*}
&
&R_{\hat{A}}^*(\alpha(u)\apr\alpha(v)-\alpha(\ell(\alpha(u))v)-\alpha(ur(\alpha(v)))
-r^*(\alpha(u))\alpha^*(b^*)+\alpha^*(R_{\hat A}^*(\alpha(u))b^*)\\
& &-\alpha^*(uL_{\hat A}^*(\alpha^*(b^*)))-
\alpha^*(a^*)\ell^*(\alpha(v))-\alpha^*(R_{\hat
A}^*(\alpha^*(a^*))v)+\alpha^*(a^*L_{\hat A}^*(\alpha(v))))w+\\ & &
R_{\hat{A}}^*(\alpha(u)\apr\alpha(v)-\alpha(\ell(\alpha(u))v)-
\alpha(ur(\alpha(v))) -r^*(\alpha(u))\alpha^*(b^*)+\alpha^*(R_{\hat
A}^*(\alpha(u))b^*)\\ & &-\alpha^*(uL_{\hat A}^*(\alpha^*(b^*)))-
\alpha^*(a^*)\ell^*(\alpha(v))-\alpha^*(R_{\hat
A}^*(\alpha^*(a^*))v)+\alpha^*(a^*L_{\hat A}^*(\alpha(v))))c^*\\
&=&uL_{\hat{A}}^*(\alpha(v)\apr\alpha(w)-\alpha(\ell(\alpha(v))w)-\alpha(vr(\alpha(w)))-
r^*(\alpha(v))\alpha^*(c^*)+\alpha^*(R_{\hat{A}}^*(\alpha(v))c^*)\\
& &-
\alpha^*(vL_{\hat{A}}^*(\alpha^*(c^*)))-\alpha^*(b^*)\ell^*(\alpha(w))-\alpha^*(R_{\hat{A}}^*(\alpha^*(b^*))w)+
\alpha^*(b^*L_{\hat{A}}^*(\alpha(w))))+\\
&
&a^*L_{\hat{A}}^*(\alpha(v)\apr\alpha(w)-\alpha(\ell(\alpha(v))w)-\alpha(vr(\alpha(w)))-
r^*(\alpha(v))\alpha^*(c^*)+\alpha^*(R_{\hat{A}}^*(\alpha(v))c^*)\\
& &-
\alpha^*(vL_{\hat{A}}^*(\alpha^*(c^*)))-\alpha^*(b^*)\ell^*(\alpha(w))-\alpha^*(R_{\hat{A}}^*(\alpha^*(b^*))w)+
\alpha^*(b^*L_{\hat{A}}^*(\alpha(w)))).
\end{eqnarray*}
}
By suitable choices of $u,v,w\in V$ and $a^*,b^*,c^*\in A^*$, we find that this equation holds if and only if the following equations hold:
{\allowdisplaybreaks \small
\begin{eqnarray}
&&R_{\hat{A}}^*(\alpha(u)\apr\alpha(v)-\alpha(\ell(\alpha(u))v+ur(\alpha(v))))w
\mlabel{eq:expan1}
\\
&=&
uL_{\hat{A}}^*(\alpha(v)\apr\alpha(w)-\alpha(\ell(\alpha(v))w+vr(\alpha(w))))
\notag \quad \text{(take $a^*=b^*=c^*=0$)},
\end{eqnarray}
\begin{eqnarray}
& &R_{\hat{A}}^*(-r^*(\alpha(u))\alpha^*(b^*)+\alpha^*(R_{\hat
A}^*(\alpha(u))b^*)-\alpha^*(uL_{\hat A}^*(\alpha^*(b^*))))w\mlabel{eq:expan2}\\
&=&
uL_{\hat{A}}^*(-\alpha^*(b^*)\ell^*(\alpha(w))-\alpha^*(R_{\hat{A}}^*(\alpha^*(b^*))w)+
\alpha^*(b^*L_{\hat{A}}^*(\alpha(w))))\notag
\quad\text{(take $v=a^*=c^*=0$)},
\end{eqnarray}
\begin{eqnarray}
& &R_{\hat{A}}^*(-\alpha^*(a^*)\ell^*(\alpha(v))-\alpha^*(R_{\hat
A}^*(\alpha^*(a^*))v)+\alpha^*(a^*L_{\hat A}^*(\alpha(v))))w\mlabel{eq:expan3}\\
&=&
a^*L_{\hat{A}}^*(\alpha(v)\apr\alpha(w) -\alpha(\ell(\alpha(v))w)-\alpha(vr(\alpha(w))))\notag
\quad\text{(take $u=b^*=c^*=0$)}, \end{eqnarray}
\begin{eqnarray}
& &R_{\hat{A}}^*(\alpha(u)\apr\alpha(v)-\alpha(\ell(\alpha(u))v)-
\alpha(ur(\alpha(v))))c^*\mlabel{eq:expan4}\\
&=&uL_{\hat{A}}^*(-
r^*(\alpha(v))\alpha^*(c^*)+\alpha^*(R_{\hat{A}}^*(\alpha(v))c^*)-
\alpha^*(vL_{\hat{A}}^*(\alpha^*(c^*))))\notag
\quad \text{(take $w=a^*=b^*=0$)},
\end{eqnarray}
\begin{eqnarray}
& &R_{\hat{A}}^*(-r^*(\alpha(u))\alpha^*(b^*)+\alpha^*(R_{\hat
A}^*(\alpha(u))b^*)-\alpha^*(uL_{\hat A}^*(\alpha^*(b^*))))c^*\mlabel{eq:expan5}\\
&=&0
\quad \text{(take $v=w=a^*=0$)},
\notag
\end{eqnarray}
\begin{eqnarray}
& &R_{\hat{A}}^*(- \alpha^*(a^*)\ell^*(\alpha(v))-\alpha^*(R_{\hat
A}^*(\alpha^*(a^*))v)+\alpha^*(a^*L_{\hat
A}^*(\alpha(v))))c^*\mlabel{eq:expan6}\\
&=&a^*L_{\hat{A}}^*(-
r^*(\alpha(v))\alpha^*(c^*)+\alpha^*(R_{\hat{A}}^*(\alpha(v))c^*)-
\alpha^*(vL_{\hat{A}}^*(\alpha^*(c^*))))\notag
\quad\text{(take $u=w=b^*=0$)},
\end{eqnarray}
\begin{eqnarray}
&
&a^*L_{\hat{A}}^*(-\alpha^*(b^*)\ell^*(\alpha(w))-\alpha^*(R_{\hat{A}}^*(\alpha^*(b^*))w)+
\alpha^*(b^*L_{\hat{A}}^*(\alpha(w))))\mlabel{eq:expan7} \\
&=& 0 \quad \text{(take $u=v=c^*=0$)}. \notag
\end{eqnarray}
}
Thus we just need to prove
\begin{enumerate}
\item
Eq.~(\mref{eq:expan1}) $\Leftrightarrow$ Eq.~(\mref{eq:biaspro}), \item
Eq.~(\mref{eq:expan2}) $\Leftrightarrow$ Eq.~(\mref{eq:uvx1}), \item
Eq.~(\mref{eq:expan3}) $\Leftrightarrow$ Eq.~(\mref{eq:uvx2}), \item
Eq.~(\mref{eq:expan4}) $\Leftrightarrow$ Eq.~(\mref{eq:uvx3}),
\item
both sides of Eq.~(\mref{eq:expan6}) equal to zero, and
\item
Eq.~(\mref{eq:expan5}) and Eq.~(\mref{eq:expan7}) hold.
\end{enumerate}
The proofs of these statements are similar. So we just prove that
Eq.~(\mref{eq:expan2}) holds if and only if Eq.~(\mref{eq:uvx1}) holds.
Let $LHS$ and $RHS$
denote the left-hand side and right-hand side of
Eq.~(\mref{eq:expan2}) respectively. Then for any $x\in A,s^*\in V^*$, we have
$$\langle LHS,s^*\rangle =\langle RHS,s^*\rangle=0.$$
Further
\begin{eqnarray*}
\langle LHS, x\rangle&=&\langle
w,-r^*(x)(r^*(\alpha(u))\alpha^*(b^*))+r^*(x)\alpha^*(R_{\hat{A}}^*(\alpha(u))b^*)-
r^*(x)\alpha^*(uL_{\hat{A}}^*(\alpha^*(b^*)))\rangle\\
&=&\langle-\alpha((wr(x))r(\alpha(u)))+\alpha(wr(x))\apr\alpha(u),b^*\rangle
-\langle\alpha^*(b^*)\apr\alpha(wr(x)),u\rangle\\
&=&\langle-\alpha((wr(x))r(\alpha(u)))+\alpha(wr(x)) \apr\alpha(u)-\alpha(\ell(\alpha(wr(x)))u),b^*\rangle,\\
\langle RHS,x\rangle&=&\langle
u,-(\alpha^*(b^*)\ell^*(\alpha(w)))\ell^*(x)-\alpha^*(R_{\hat{A}}^*(\alpha^*(b^*))w)\ell^*(x)
+\alpha^*(b^*L_{\hat{A}}^*(\alpha(w)))\ell^*(x)\rangle\\
&=&\langle-\alpha(\ell(\alpha(w))(\ell(x)u)),b^*\rangle-\langle\alpha(\ell(x)u)\apr\alpha^*(b^*),w\rangle+
\langle\alpha(w)\apr\alpha(\ell(x)w),b^*\rangle\\
&=&\langle-\alpha(\ell(\alpha(w))(\ell(x)u))-\alpha(wr(\alpha(\ell(x)u)))+\alpha(w)\apr\alpha(\ell(x)u),b^*\rangle.
\end{eqnarray*}
So Eq.~(\mref{eq:expan2}) holds if and only if Eq.~(\mref{eq:uvx1})
holds.
\end{proof}

\begin{coro}
Let $(A,\apr)$ be a $\bfk$-algebra with finite $\bfk$-dimension.
\begin{enumerate}
\item Let $(R, \rpr, \ell,r)$ be an $A$-bimodule $\bfk$-algebra with finite $\bfk$-dimension.
Let $\alpha,\beta:R\to A$ be two linear maps such that $\alpha$ is
an \tto $\calo$-operator of weight $\lambda$ with \bop $\beta$ of
\bwt $(\kappa,\mu)$, i.e., $\beta$ is an $A$-bimodule homomorphism
and the conditions $($\mref{eq:ksy}$)$ and $($\mref{eq:mueq}$)$ in
Definition~\mref{de:co} hold, and $\alpha$ and $\beta$ satisfy
Eq.~$($\mref{eq:gmybe}$)$. Then $\alpha-\alpha^{21}$ identified as an
element of
$(A\ltimes_{r^*,\ell^*}R^*)\otimes(A\ltimes_{r^*,\ell^*}R^*)$ is a
skew-symmetric solution of the \GAYBE~$($\mref{eq:maybe}$)$ if and only
if the following equations hold:
\begin{equation}
\lambda \ell(\alpha(u\rpr v))w=\lambda ur(\alpha(v\rpr w)),\quad
\forall u,v,w\in R,\mlabel{eq:lambdakmucon1}
\end{equation}
\begin{equation}
\lambda\alpha(u(vr(x)))=\lambda\alpha(u\rpr v)\apr x,\quad \forall
u,v\in R,x\in A,\mlabel{eq:lambdakmucon2}
\end{equation}
\begin{equation}
\lambda\alpha((\ell(x)u)\rpr v)=\lambda x\apr\alpha(u\rpr v),\quad
\forall u,v\in R,x\in A.\mlabel{eq:lambdakmucon3}
\end{equation}
In particular, when $\lambda=0$, i.e., $\alpha$ is an \tto
$\calo$-operator of weight zero with \bop $\beta$ of \bwt $(\kappa ,\mu)$,
$\alpha-\alpha^{21}$ identified as an element of
$(A\ltimes_{r^*,\ell^*}R^*)\otimes(A\ltimes_{r^*,\ell^*}R^*)$ is a
skew-symmetric solution of the \GAYBE~$($\mref{eq:maybe}$)$.
\mlabel{it:mybeco1}
\item Let $(R,\rpr, \ell,r)$ be an
$A$-bimodule $\bfk$-algebra with finite $\bfk$-dimension. Let $\alpha:R\to A$ be an
$\calo$-operator of weight $\lambda$. Then $\alpha-\alpha^{21}$
identified as an element of
$(A\ltimes_{r^*,\ell^*}R^*)\otimes(A\ltimes_{r^*,\ell^*}R^*)$ is a
skew-symmetric solution of the \GAYBE if and only if
Eq.~$($\mref{eq:lambdakmucon1}$)$, Eq.~$($\mref{eq:lambdakmucon2}$)$ and
Eq.~$($\mref{eq:lambdakmucon3}$)$ hold.\mlabel{it:mybeco2}
\item Let
$(V,\ell,r)$ be a bimodule of $A$ with finite $\bfk$-dimension. Let $\alpha,\beta:V\to A$ be
two linear maps such that $\alpha$ is an \tto $\calo$-operator
with \bop $\beta$ of \bwt $\kappa$. Then $\alpha-\alpha^{21}$
identified as an element of
$(A\ltimes_{r^*,\ell^*}V^*)\otimes(A\ltimes_{r^*,\ell^*}V^*)$ is a
skew-symmetric solution of the \GAYBE.\mlabel{it:mybeco3}
\item Let
$\alpha:A\to A$ be a linear endomorphism of $A$. Suppose that
$\alpha$ satisfies Eq.~$($\mref{eq:kmyb}$)$. Then $\alpha-\alpha^{21}$
identified as an element of
$(A\ltimes_{R^*,L^*}A^*)\otimes(A\ltimes_{R^*,L^*}A^*)$ is a
skew-symmetric solution of the \GAYBE.\mlabel{it:mybeco4}
\item Let
$(R,\rpr, \ell,r)$ be an $A$-bimodule $\bfk$-algebra of finite $\bfk$-dimension. Let
$\alpha,\beta:R\to A$ be two linear maps such that $\alpha$ is an
\tto $\calo$-operator with \bop $\beta$ of \bwt
$(\kappa,\mu)=(0,\mu)$, i.e., $\beta$ is an $A$-bimodule
homomorphism and the condition $($\mref{eq:mueq}$)$ in
Definition~\mref{de:co} holds, and $\alpha$ and $\beta$ satisfy
the following equation:
\begin{equation}
\alpha(u)\apr\alpha(v)-\alpha(\ell(\alpha(u))v+ur(\alpha(v)))=\mu\beta(u\rpr
v),\quad \forall u,v\in R.
\notag
\end{equation}
Then $\alpha-\alpha^{21}$ identified as an element of
$(A\ltimes_{r^*,\ell^*}R^*)\otimes(A\ltimes_{r^*,\ell^*}R^*)$ is a
skew-symmetric solution of the \GAYBE.\mlabel{it:mybeco5}
\end{enumerate}
\mlabel{co:motoaybe2}
\end{coro}
\begin{proof}
(\mref{it:mybeco1}) Since $\alpha$ is an \tto $\calo$-operator of
weight $\lambda$ with \bop $\beta$ of \bwt $(\kappa,\mu)$, by
Theorem~\mref{thm:maybeequi}, $\alpha-\alpha^{21}$ identified as
an element of
$(A\ltimes_{r^*,\ell^*}R^*)\otimes(A\ltimes_{r^*,\ell^*}R^*)$ is a
skew-symmetric solution of the \GAYBE~(\mref{eq:maybe}) if and only
if the following equations hold:
\begin{eqnarray}
&&-\lambda \ell(\alpha(u\rpr v))w+\kappa \ell(\beta(u)\apr\beta(v))w+\mu
\ell(\beta(u\rpr v))w\mlabel{eq:mybecoro1}\\
&=& -\lambda ur(\alpha(v\rpr
w))+\kappa ur(\beta(v)\apr\beta(w))+\mu ur(\beta(v\rpr
w)), \notag
\end{eqnarray}
\begin{eqnarray}
&&-\lambda\alpha((ur(x))\rpr
v)+\kappa \beta(ur(x))\apr\beta(v)+\mu\beta((ur(x))\rpr v)\mlabel{eq:mybecoro2}\\
&=& -\lambda\alpha(u\rpr(l(x)v))+\kappa \beta(u)\apr\beta(\ell(x)v) +\mu\beta(u\rpr(\ell(x)v)),
\notag
\end{eqnarray}
\begin{eqnarray}
&&-\lambda\alpha(u\rpr(vr(x)))+\kappa \beta(u)\apr\beta(vr(x)) +\mu\beta(u\rpr(vr(x)))\mlabel{eq:mybecoro3}\\
&=&
-\lambda\alpha(u\rpr v)\apr x+\kappa (\beta(u)\apr\beta(v))\apr
x+\mu\beta(u\rpr v)\apr x,\notag
\end{eqnarray}
\begin{eqnarray}
&&-\lambda\alpha((\ell(x)u)\rpr
v)+\kappa \beta(\ell(x)u)\apr\beta(v)+\mu\beta((\ell(x)u)\rpr v)\mlabel{eq:mybecoro4}\\
&=& -\lambda
x\apr\alpha(u\rpr v)+\kappa x\apr(\beta(u)\apr\beta(v))+\mu
x\apr\beta(u\rpr v),\notag
\end{eqnarray}
for any $u,v\in R,x\in A$. Since $\beta$ is an $A$-bimodule
homomorphism and the conditions (\mref{eq:ksy}) and
(\mref{eq:mueq}) in Definition~\mref{de:co}
hold, we have Eq.~(\mref{eq:lambdakmucon1}) holds if and only if
Eq.~(\mref{eq:mybecoro1}) holds, Eq.~(\mref{eq:lambdakmucon2})
holds if and only if Eq.~(\mref{eq:mybecoro3}) holds,
Eq.~(\mref{eq:lambdakmucon3}) holds if and only if
Eq.~(\mref{eq:mybecoro4}) holds and Eq.~(\mref{eq:mybecoro2})
holds automatically.
\medskip

\noindent (\mref{it:mybeco2}) This follows from
Item~(\mref{it:mybeco1}) by setting $\kappa =\mu=0$.
\medskip

\noindent (\mref{it:mybeco3}) This follows from
Item~(\mref{it:mybeco1}) by setting $\lambda=\mu=0$.
\medskip

\noindent (\mref{it:mybeco4}) This follows from
       Item~(\mref{it:mybeco3}) for $(V,\ell,r)=(A,L,R)$ and $\beta={\rm
       id}$.
\medskip

\noindent (\mref{it:mybeco5}) This follows from
Item~(\mref{it:mybeco1}) by setting $\lambda=\kappa =0$.
\end{proof}


%
%

\end{document}